\documentclass[11pt, reqno]{amsart}
\usepackage{amssymb}
\usepackage{graphicx}
\usepackage{physics}
\usepackage{amssymb}
 \usepackage{amsthm}
\usepackage{epsf}
\usepackage{amsbsy,amsmath}
\usepackage{mathtools}
\usepackage{upgreek}
\usepackage{mathrsfs}
\usepackage{amsfonts}
\usepackage{amssymb}
\usepackage{float}
\usepackage{enumitem}
\usepackage{eucal}
\usepackage{graphics,mathrsfs}
\usepackage{amsthm}
\usepackage{secdot}
\usepackage{esint}
\usepackage{hyperref}
\usepackage{varwidth}
\usepackage{tasks}
\usepackage{abstract}
\theoremstyle{plain}
\newtheorem{theorem}{Theorem}[section]
\newtheorem{cor}{Corollary}[section]

\theoremstyle{definition}
\newtheorem{definition}{Definition}[section] % This numbers definitions based on sections
\newtheorem{lemma}{Lemma}[section]
\newtheorem{remark}{Remark}[section]

\allowdisplaybreaks

\usepackage{xcolor}
\usepackage{bbm}
\usepackage{cite}

\renewcommand{\d}{\:\! \mathrm{d}}
\long\def\symbolfootnote[#1]#2{\begingroup
\def\thefootnote{\fnsymbol{footnote}}\footnote[#1]{#2}\endgroup}
\numberwithin{equation}{section}
\begin{document}
\title[Mixed local and nonlocal equations involving singularity and nonregular data] {Existence results for mixed local and nonlocal elliptic equations involving singularity and nonregular data}
\author{Souvik Bhowmick and Sekhar Ghosh}
\address[Souvik Bhowmick]{Department of Mathematics, National Institute of Technology Calicut, Kozhikode, Kerala, India - 673601}
\email{souvikbhowmick2912@gmail.com / souvik\_p230197ma@nitc.ac.in}
\address[ Sekhar Ghosh]{Department of Mathematics, National Institute of Technology Calicut, Kozhikode, Kerala, India - 673601}
\email{sekharghosh1234@gmail.com / sekharghosh@nitc.ac.in}
%\thanks{{\em 2020 Mathematics Subject Classification: } 35R11, 35J75, 35J20}
%\keywords{Mixed operators, Singularity, {\it{duality}} solution, Weak and veryweak solution, Kato-type inequality.}
\maketitle
\begin{abstract}
In this paper, we prove the existence of weak, {\it{veryweak}} and {\it{duality}} solutions to a class of elliptic problems involving singularity and measure data which is given by: $-\Delta u+(-\Delta)^s  u = \frac{f(x)}{u^\gamma} +\mu$ in $\Omega$ with the zero Dirichlet boundary data $u=0$ in $\mathbb R^N \setminus \Omega$. The existence of weak solutions is obtained by approximating a sequence of problems for $0<\gamma\leq1$ and $\gamma>1$. We employ Schauder's fixed point theorem and embeddings of Marcinkiewicz spaces. The novelty of our work is that we prove the existence of a {\it{duality}} solution and its equivalence with weak solutions to the problem $\mathcal{L}u=\mu$. Moreover, we prove a {\it{veryweak}} maximum principle and a Kato-type inequality for the mixed local-nonlocal operator $\mathcal{L}=-\Delta +(-\Delta)^s$, which are crucial tools to guarantee the existence of {\it{veryweak}} solutions to the problem. Using a Kato-type inequality, maximum principle together with sub-super solution method, we prove the existence of {\it{veryweak}} solution for $0<\gamma<1$. Our work extends the studies due to Oliva and Petitta [ESAIM Control Optim. Calc. Var., 22(1):289--308, 2016.] and Petitta [Adv. Nonlinear Stud., 16(1):115–124, 2016.] for the mixed local-nonlocal operator.
\end{abstract}
%\tableofcontents

\section{Introduction and main results  Application}\label{sec1}
In this paper, we study the following mixed local-nonlocal elliptic equation involving singularity and measure data.
\begin{equation}
\begin{aligned}
    -\Delta u +(-\Delta)^s  u&=\frac{f(x)}{u^{\gamma}}+\mu   \text{ in } \,\,\Omega,\nonumber\\
     u&>0\,\,  \text{in} \,\,\Omega,\nonumber\\
      u &=0 \,\text{in} \, \mathbb R^N \setminus \Omega,\nonumber
   \end{aligned}
    \end{equation}
    where $\Omega\subset\mathbb{R}^N$ is a bounded domain with Lipschitz boundary $\partial \Omega$, $0<s<1$, $\gamma>0$, $N>2s$, $0<f\in L^m(\Omega)$ for $m>1$ and $\mu$ is a nonnegative, bounded Radon measure. Here the operator $-\Delta +(-\Delta)^s$ is the mixed local-nonlocal elliptic operator, where $(-\Delta)^s$ is the standard fractional Laplacian defined as 
   \begin{equation*}
       (-\Delta)^s u(x)=C_{N,s}\int_{\mathbb R^N} \frac{u(x)-u(y)}{|x-y|^{N+2s}} \d y.
   \end{equation*}
   The above integral is defined as the principle value and $C_{N,s}$ is a normalising constant.

For a generation of researchers studying elliptic PDEs involving singularity the study due to Crandall {\it{et al.}}\cite{CRT1977} and Lazer and Mckenna \cite{LM1991} still remains one of the pioneering works. In \cite{CRT1977}, the authors proved the existence of a solution whereas in \cite{LM1991}, the existence of a $H_0^1$ and $C^1$ solution is established for the following problem 
    \begin{equation}\label{eq1.1 A}
    \begin{aligned}
        -\Delta u & =\frac{f(x)}{u^\gamma}   \text{ in } \,\,\Omega,\\
     u&>0\,\,  \text{in} \,\,\Omega,\\
    u & = 0  \text { on } \partial \Omega,
    \end{aligned}
    \end{equation}
    where $\gamma>0$, $\Omega\subset\mathbb R^N$ with Lipschitz boundary $\partial \Omega$ and $f$ is a nonnegative function over $\Omega$. The celebrated result of \cite{LM1991} is that if $f(x)\in C^\alpha(\bar \Omega)$, $0<\alpha<1$, then one can obtain a unique solution (classical) $u\in C^{2+\alpha}(\Omega) \cap C(\bar \Omega)$ by constructing appropriate subsolution and supersolution to it. Moreover, $u\in H_0^1(\Omega)$ if and only if $\gamma<3$ and if $\gamma>1$ then $u$ fails to be in $C^1(\bar \Omega)$. Boccardo and Orsina \cite{BO2010} impelled the variational tools to obtain the existence and uniqueness results on approximating problems. Yijing and Duanzhi \cite{YD2014} proved that whenever $f\in L^1(\Omega)$ and $\gamma>1$, the problem \eqref{eq1.1 A} admits a unique $H^1_0(\Omega)$ solution if and only if $\int_\Omega fu^{1-\gamma}\d x < +\infty$ for some there $u\in H^1_0(\Omega)$. Furthermore, if $\gamma \geq 3$ then $\int_\Omega u^{1-\gamma}\d x = +\infty$ for all $u\in H_0^1(\Omega)$. Recently, Oliva and Petitta presented an excellent review on the progress of singular problems involving Laplacian \cite{OP2024}. For the existence result of a singular problem involving the $p$-Laplacian, we refer to Canino {\it{et al.}} \cite{CST2016}.

     The nonlocal counterpart to the problem \eqref{eq1.1 A} is given by 
    \begin{equation}\label{eq1.2 B}
    \begin{aligned}
       (-\Delta)^s u & =\frac{f(x)}{u^\gamma}   \text{ in } \,\,\Omega,\\
     u&>0\,\,  \text{in} \,\,\Omega,\\
    u & =0 \text{ on } \mathbb R^N \setminus \Omega,
    \end{aligned}
    \end{equation}
    where $\gamma >0$, $0<s<1$, $\Omega\subset\mathbb R^N$  is bounded domain with Lipschitz boundary $\partial \Omega$ and $f$ is a nonnegative on $\Omega$. Tai and Fang \cite{F2014} studied the problem \eqref{eq1.2 B} when $f(x)=1$ which has further been investigated by Youssfi and Mahmoud \cite{YM2021}. We refer the reader to Canino {\it{et al.}} \cite{CMSS2017} for the existence result for the nonlinear operator, namely the fractional $p$-Laplacian. In \cite{CMSS2017}, the authors established the result by dividing the singular exponent $\gamma$ into three favourable cases, which are $0<\gamma<1,\,\gamma=1$ and $\gamma>1$. Afterwards, the study of singular PDEs involving elliptic operators evolved significantly. It is almost impossible to enlist them here. For a well-acquainted development in the existence theory in the literature on singular problems, we refer to  \cite{T1979, AGM1998, CD2004, C2013, AM2014, CS2016, GMM2017, PGC2017, OP2018, G2018, H2003, BBMP2015, AGS2018, S2017, SGC2019} and the references cited therein.

Another, important phenomena in the study of PDEs occur when the given data is a measure. In general, the existence of a weak solution fail even in the case of a Dirac measure $\delta_a$. The result is addressed in Brezis {\it{et al.}} \cite{BMP2007}, where they considered the following PDE involving measure,
    \begin{align}\label{eq1.3 C}
    -\Delta u + g(x,u) &= \mu ~ \,\, \text{in} \,\,\, \Omega,\\ \nonumber
    u & = 0 ~ \,\, \text{on} \,\,\, \partial\Omega,
\end{align}
where $\Omega \subset \mathbb R^N$ be a bounded domain with smooth boundary and $\mu$ is a nonnegative bounded Radon measure. In \cite{BMP2007}, the authors proved that for $g(x,u)=|u|^{p-1}u$, the problem \eqref{eq1.3 C} has a unique weak solution in $W_0^{1,q}(\Omega)$ for $q<\frac{N}{N-1}.$ Moreover,  if $N \geq 3$ and $ p \geq \frac{N}{N-2}$ then there does not exists any solution even for $\mu=\delta_a$ and if $N \geq 2$ and $p < \frac{N}{N-2}$ then there always exists a solution for any nonnegative measure $\mu$. For $g(x,u)=0$, the problem \eqref{eq1.3 C} is studied by Ponce \cite{A2016}, where $\mu$ is a finite Borel measure. For the nonlinear case with the $p$-Laplacian, we refer to Kilpel{\"a}inen \cite{K2002}. 
 
The following nonlocal problem with measure data is studied by Petitta \cite{P2016}.
\begin{align}\label{eq1.3 D}
    (-\Delta)^s u &= \mu ~ \,\, \text{in} \,\,\, \Omega,\\ \nonumber
    u & = 0 ~ \,\, \text{on} \,\,\, \partial\Omega,
\end{align}
where $\Omega \subset \mathbb R^N$ be a bounded domain with smooth boundary, $0<s<1$ and $\mu$ is a nonnegative bounded Radon measure in $\Omega$. In \cite{P2016}, the author established the existence of duality solution and its equivalence with the weak solution to \eqref{eq1.3 D}. The study of a {\it{duality}} solution to \eqref{eq1.3 D} when $\Omega=\mathbb{R}^N$ and $\frac{1}{2}<s<1$, can be found in Karlsen {\it{et al.}} \cite{KPU2011}. For the existence of a weak solution involving the fractional $p$-Laplacian, we refer to  Abdellaoui {\it{et al.}} \cite{AAB2016}. For a well-presented development and related pioneering contribution, we refer to \cite{MV2014, V2004, KMS2015, BMP2007, P2016, KPU2011, A2016, K2002, DMOP1999, AAB2016} and the references cited therein.

\par We now turn our attention to the existence theory of elliptic problems involving both singularity and measure data. The following Dirichlet problem involving a singular term and measure data was investigated by Oliva and Petitta \cite{OP2016}.
\begin{equation}\label{eq1.4 D}
    \begin{aligned}
        -\Delta u & =\frac{f(x)}{u^\gamma}+\mu   \text{ in } \,\,\Omega,\\
     u&>0\,\,  \text{in} \,\,\Omega,\\
    u & = 0  \text { on } \partial \Omega,
    \end{aligned}
    \end{equation}
    where $\gamma> 0$, $\Omega \subset \mathbb{R}^N$ be a open Lipschitz bounded domain, $f$ is a nonnegative function and $\mu$ is a nonnegative, bounded Radon measure. In \cite{OP2016}, the authors first considered a sequence of problems and proved the existence of a weak solution to \eqref{eq1.4 D} for $0<\gamma\leq1$ and $\gamma>1$ by using the Schauder's fixed point theorem. Moreover, with the help of a Kato-type inequality and sub-super solution method, they guaranteed the existence of a $\it{veryweak}$ solution. The nonlocal counterpart of the work due to Oliva and Petitta \cite{OP2016} is studied by Ghosh {\it{et al.}} \cite{GCG2019} in the nonlocal setting. In \cite{GCG2019}, the authors considered the following problem.
    In the nonlocal case,
    \begin{equation}\label{eq1.5 E}
    \begin{aligned}
       (-\Delta)^s u & =\frac{f(x)}{u^\gamma}+\mu  \text{ in } \,\,\Omega,\\
     u&>0\,\,  \text{in} \,\,\Omega,\\
    u & =0 \text{ on } \mathbb R^N \setminus \Omega,
    \end{aligned}
    \end{equation}
    where $\gamma >0$, $0<s<1$, $\Omega\subset\mathbb R^N$ is bounded domain with Lipschitz boundary $\partial \Omega$, $f$ is a nonnegative function and $\mu$ is a nonnegative, bounded Radon measure. On using the weak convergence method on an appropriate sequence of PDEs, the authors \cite{GCG2019} guaranteed the existence of a weak solution to the problem \eqref{eq1.5 E} for $0<\gamma\leq1$ and $\gamma>1$. Further, they proved a nonlocal Kato-type inequality and employed it with the sub-super solution method to establish the existence of a $\it{veryweak}$ solution. For further studies in this direction, we refer to \cite{BH2021, OP2016, OP2024, PGC2017, GCG2019} and the references therein.

    \par Recently, elliptic PDEs involving mixed local and nonlocal elliptic operators have received a lot of attention. Such problems arise in the biological modelling viz. population dynamics \cite{DV2021}, Brownian motion and L\'evy process \cite{DPV2023}. For a detailed study of the development from an application point of view, we refer to \cite{DV2021, DPV2023, BDVV2022} and the references therein. Consider the following problem involving both the local and nonlocal operators. 
    \begin{equation}\label{eq1.6 F}
             \begin{aligned}
                 -\Delta_p u + (-\Delta_p)^s u & = \frac{f(x)}{u^\gamma}   \text{ in } \,\,\Omega,\\
     u&>0\,\,  \text{in} \,\,\Omega,\\
    u &=0 \text{ on } \mathbb R^N \setminus \Omega,
                 \end{aligned}
         \end{equation}  
         where $\Omega\subset\mathbb R^N$ is bounded domain with Lipschitz boundary $\partial \Omega$, $0<s<1<p<\infty$, $f$ is a nonnegative function and $\mu$ is a nonnegative, bounded Radon measure. For $p=2$, Arora and R{\u{a}}dulescu \cite{AR2023} studied the existence and nonexistence of solutions to the problem \eqref{eq1.6 F}. Garain and Ukhlov \cite{GU2022} investigated the existence result of the problem \eqref{eq1.6 F} for $1<p<\infty$. For a comprehensive list of references on the study of mixed local-nonlocal operators, we refer to \cite{BV2023, DFV2024, G2023, BDVV2021, DM2022, GK2022, LGG2024} and references cited therein.

Before proceeding further, we would now like to present the following table of references to the readers which led to the consideration of the problem.

\begin{table}[H]
    \centering
\begin{tabular}{|c|c|c|c|}
 \hline
 Operators $(\mathcal{L})$ & Singularity & Measure data & Singularly    \\
  & &&and measure\\
  [1.0ex]
  \hline
 $(-\Delta_p)$ &\, \cite{LM1991, CST2016, CRT1977} &\, \cite{A2016, BMP2007, K2002, DMOP1999} &\, \cite{OP2016, PGC2017} \\
 &&&\\
 [1.0ex]
 \hline
 $(-\Delta_p)^s$ & \cite{F2014, CMSS2017, OP2018} & \cite{P2016, KMS2015, AAB2016, KPU2011} &\cite{GCG2019, BH2021}\\
 &&& \\
 [1.0ex]
 \hline
  $(-\Delta_p)+(-\Delta_p)^s$ & \cite{AR2023, GU2022} & \cite{CSYZ2024, BS2023} &  --\\
  &&&
  \\[1.0ex]
 \hline
\end{tabular}
\caption{}\label{table:1}
\end{table}

\iffalse

\begin{table}[H]
    \centering
\begin{tabular}{|c|c|c|c|c|}
 \hline
 Operators $(\mathcal{L})$ & Singularity & Measure data & Singularly   & Singularity and \\
  & &&and measure& Other nonlinearity \\
  [1.0ex]
  \hline
 $(-\Delta_p)$ &\, \cite{LM1991, CST2016, CRT1977} &\, \cite{A2016, BMP2007, K2002, DMOP1999} &\, \cite{OP2016, PGC2017} &\,\cite{H2003, BG2020} and the\\
 &&&&references therein.\\
 [1.0ex]
 \hline
 $(-\Delta_p)^s$ & \cite{F2014, CMSS2017, OP2018} & \cite{P2016, KMS2015, AAB2016, KPU2011} &\cite{GCG2019, BH2021}& \cite{S2017, SGC2019} and the\\
 &&&&references therein. \\
 [1.0ex]
 \hline
  $(-\Delta_p)+(-\Delta_p)^s$ & \cite{AR2023, GU2022} & \cite{CSYZ2024, BS2023} &  --& \cite{DFV2024, G2023} and the\\
  &&&&references therein.
  \\[1.0ex]
 \hline
\end{tabular}
\caption{1}\label{table:2}
\end{table}

\fi

\par Motivated by the above-mentioned studies, we consider the following Dirichlet problem involving the mixed local-nonlocal linear operator with singularity and measure data. 
         \begin{equation}
          \begin{aligned}\label{eq2:problem}
   -\Delta u+(-\Delta)^s u &= \frac{f(x)}{u^\gamma} +\mu   \text{ in } \,\,\Omega,\\
     u&>0\,\,  \text{in} \,\,\Omega,\\
  u&=0\,\, \text{in}\,\,\mathbb  R^N \setminus \Omega,
   \end{aligned}\tag{P}
\end{equation}
where $\Omega\subset\mathbb R^N$ is a bounded domain with Lipschitz boundary $\partial \Omega$, $N \geq 2$, $0 < \gamma $, $0<s<1$, $f$ is a nonnegative function and $\mu$ is a nonnegative, bounded Radon measure. We first prove the existence of a weak solution to the problem \eqref{eq2:problem} by dividing the singular exponent into two cases: $0<\gamma\leq1$ and $\gamma>1$. We employ Schauder's fixed point theorem and the weak convergence method to establish the existence result. The novelty of our work is that we prove the existence of a {\it{duality}} solution and its equivalence with weak solutions. Further, proving a {\it{veryweak}} maximum principle and a Kato-type inequality for the mixed local-nonlocal operator, we prove the existence of {\it{veryweak}} solutions to the problem \eqref{eq2:problem}. In our forthcoming paper, we will study the nonlinear extension to the problem \eqref{eq2:problem}, i.e. when $p\in(1,\infty)$.\\

\noindent We now state the first result of our work.
\begin{theorem}\label{t1.1}
Let $\Omega\subset\mathbb{R}^N$ be a bounded domain with Lipschitz boundary $\partial\Omega$, $s\in(0,1)$ and let $0<f\in L^1(\Omega)$. When $0<\gamma \leq 1$, there exists a weak solution $u\in X_0^{s,q}(\Omega)$ to the problem \eqref{eq2:problem} for every $ q < \frac{N}{N-1}$. If $\gamma > 1$, there exists a weak solution $u\in X_{loc}^{s,q}(\Omega)$ to the problem \eqref{eq2:problem} for every $ q < \frac{N}{N-1}$. (Refer to Section \ref{sec3}).
\end{theorem}
Our next aim is to obtain the existence of {\it{veryweak}} solutions to the problem \eqref{eq2:problem}. We first prove the existence of {\it{veryweak}} solutions to the problem \eqref{eq1.3E} by proving its equivalence with {\it{duality}} solutions for the following problem.  
\begin{align}\label{eq1.3E}
   (-\Delta)u+(-\Delta)^s u &= \mu ~ \,\, \text{in} \,\,\, \Omega,\\ \nonumber
    u & = 0 ~ \,\, \text{on} \,\,\, \partial\Omega,
\end{align}
where $\Omega \subset \mathbb R^N$ is a bounded domain with Lipschitz boundary $\partial\Omega$. The literature for the problem \eqref{eq1.3E} is limited. Byun and Song \cite{BS2023} investigate the existence of weak solutions taking $\mu$ as a signed  Borel measure with a finite total mass on $\mathbb R^N$. When $\mu$ is a bounded Borel measure, the existence is established in \cite{CSYZ2024}. To the best of our knowledge, the existence of {\it{duality}} solutions and their equivalence with the weak solutions to the problem \eqref{eq1.3E} is not available in the literature. We now state our second result of this work.
\begin{theorem}\label{t1.2}
 Let $\Omega\subset\mathbb{R}^N$ be a bounded domain with Lipschitz boundary $\partial\Omega$ and $s\in(0,1)$, $\gamma\in(0,1)$. Then there exists a unique $\it{duality~solution}$ $u\in L^1(\Omega)$ to the problem \eqref{eq1.3E} in the sense of Definition \ref{dual def}. In addition, $u\in L^q(\Omega)$ for $q\in(1,\frac{N}{N-2})$ if $s\in(0,\frac{1}{2}]$ and $q\in(\frac{N}{N-2s+1},\frac{N}{N-1})$ if $s\in(\frac{1}{2},1).$ Moreover, every {\it{duality}} solution to the problem \eqref{eq1.3E} is a weak solution to the problem \eqref{eq1.3E} and conversely. (See Section \ref{sec4}).
\end{theorem}
It is noteworthy to mention here that the results of Theorem \ref{t1.2} hold even if we replace $\mu$ by $f+\mu$, where $f\in L^1(\Omega)$. We now proceed to the final result of the paper, that is we will prove the existence of a {\it{veryweak}} solution to the problem \eqref{eq2:problem}. We follow \cite{OP2016}, where the authors obtained the existence of solutions by applying the sub-super solution method combined with a Kato-type inequality \cite{BMP2007}. For the fractional Kato-type inequality, we refer \cite{LPPS2015}. The main difficulty in our case is the nonavailability of Kato-type inequality. We have proved a Kato-type inequality for the mixed local-nonlocal operator (see Theorem \ref{thm6.1}, \eqref{keto inq 2}). Further, we proved a {\it{veryweak}} maximum principle (see Lemma \ref{eq6.15A}) for the mixed local and nonlocal operator $ -\Delta +(-\Delta)^s$.
The weak maximum principle is established by Biagi et al. \cite[Theorem 1.2]{BDVV2022}. The next result states the existence of {\it{veryweak}} solution to the problem \eqref{eq2:problem}.
\begin{theorem}\label{t1.3}
       Let $\Omega\subset\mathbb{R}^N$ be a bounded domain with Lipschitz boundary $\partial\Omega$ and $s\in(0,1)$, $\gamma\in(0,1)$. Let $f>0$ in $\Bar{\Omega}$ such that $f\in C^{\beta}(\Bar{\Omega})$, then the problem \eqref{eq2:problem} possesses a $\it{veryweak}$ solution in the sense of the Definition \ref{def5.1 VW}. In addition, if $0<f\in L^1(\Omega)\cap L^\infty(\Omega_\eta)$, then the problem \eqref{eq2:problem} possesses a $\it{veryweak}$ solution in the sense of the Definition \ref{def5.1 VW}. (Refer to Section \ref{sec5}). 
\end{theorem}

The rest of the paper is organized as follows: In Section \ref{sec2}, we recall some fundamental results and develop the necessary tools for the solution space related to our problem. Section \ref{sec3} is devoted to obtaining the existence of weak solutions to the problem considered. In Section \ref{sec4}, we establish the existence of a {\it{duality}} solution and its equivalence with a weak solution. Section \ref{sec5} is engrossed with the existence result for the existence of a $\it{veryweak}$ solution. Finally, in Section \ref{sec6}, we prove a Kato-type inequality for the mixed local and nonlocal operator, a {\it{veryweak}} maximum principle and a weak comparison principle for the mixed operator. 
\section{Preliminaries and solution space setup}\label{sec2}
\noindent We begin this section with the definitions and fundamental properties of Sobolev and fractional Sobolev spaces. Let $\Omega\subset\mathbb{R}^N,\, N\geq2$ be a bounded domain with Lipschitz boundary $\partial\Omega$. 
   \begin{definition}
       For each integer $k\geq0$ and $p\in[1, \infty)$, the Sobolev space $W^{k,p}(\Omega)$ is defined as follows:
   \begin{align*}
       W^{k,p}(\Omega)=\{u \in L^p(\Omega):D^\alpha u \in L^p(\Omega),\,\, \forall\, |\alpha| \leq k\},
       \end{align*}
       where $\alpha$ is a multi-index. The space $W^{k,p}(\Omega)$ is a is a Banach space equipped with the norm 
       \begin{equation}\label{n1}
           \|u\|_{W^{k,p}(\Omega)}=\left(\sum_{|\alpha| \leq k} \|D^\alpha u\|^p_{L^p(\Omega)}\right)^\frac{1}{p}.
       \end{equation} 
       Moreover, it is separable for $1\leq p<\infty$ and reflexive for $1<p<\infty.$ Note that for $k=1$, the norm \eqref{n1} is equivalent to the following norm,
       \begin{equation}\label{n1-1}
           \|u\|_{W^{1,p}(\Omega)}=\|u\|_{L^p(\Omega)}+ \|\nabla u\|_{L^p(\Omega)}. 
       \end{equation}
   \end{definition}
\noindent We follow Di Nezza \textit{et al.} \cite{NPV2012} for the fractional Sobolev spaces, which is defined as   
   \begin{definition}
       For every $0<s<1\leq p<\infty$, the fractional Sobolev space $W^{s,p} (\Omega)$ is defined as
       \begin{equation*}
       \begin{aligned}
           W^{s,p}(\Omega)&=\left\{u \in L^p(\Omega): \frac{|u(x)-u(y)|}{|x-y|^{\frac{N}{p}+s}} \in L^p(\Omega \times \Omega)\right\}.
            \end{aligned}
           \end{equation*}
           The space $W^{s,p} (\Omega)$ is a Banach space for all $p\in[1,\infty)$, equipped with the norm
           \begin{equation}\label{n s}
               \|u\|_{W^{s,p}(\Omega)}=\left(\|u\|^p_{L^p(\Omega)}+ [u]_{W^{s,p}(\Omega)}^p\right)^{\frac{1}{p}},
           \end{equation}
           where $[u]_{W^{s,p}(\Omega)}$ is the Gagliardo seminorm of $u$ which is given by
           \begin{equation}\label{n g}
               [u]_{W^{s,p}(\Omega)}=\bigg(\int_\Omega \int_\Omega \frac{|u(x)-u(y)|^p}{|x-y|^{N+ps}}\d x \d y\bigg)^\frac{1}{p}.
           \end{equation}
              \end{definition}
\noindent Note that the norm \eqref{n s} is equivalent to the following norm,
          \begin{align}\label{n s1}
              \|u\|_{W^{s,p}(\Omega)}=\|u\|_{L^p(\Omega)}+[u]_{W^{s,p}(\Omega)}
          \end{align}
   Moreover, the space $W^{s,p} (\Omega)$ is separable for $1\leq p<\infty$ and reflexive for $1<p<\infty.$ Since, we are considering problems with the Dirichlet boundary data, we define the following closed subspaces of  $W^{1,p} (\Omega)$ and  $W^{s,p} (\Omega)$ respectively.
  \begin{definition}
    The space $W_0^{1,p}(\Omega)$ is defined as the closure of $ C_c^\infty(\Omega)$ in $W^{1,p}(\Omega)$ with respect to the norm in \eqref{n1-1} and the space $W_0^{s,p}(\Omega)$ is defined as the closure of $ C_c^\infty(\Omega)$ in $W^{1,p}(\Omega)$ with respect to the norm in \eqref{n s}, where  $C_c^\infty(\Omega)$ is the set of all smooth functions having compact support in $\Omega$. The spaces $W_0^{1,p}(\Omega)$ and $W_0^{s,p}(\Omega)$ are Banach spaces and are characterized as:
    \begin{align*}
        W_0^{1,p}(\Omega)&=\{u\in W^{1,p}(\Omega):u=0~\text{on}~ \partial \Omega\}\\
        W_0^{s,p}(\Omega)&=\{u\in W^{s,p}(\Omega):u=0~\text{in}~ \mathbb R^N \setminus \Omega\}.
    \end{align*}
\end{definition}
\noindent It is noteworthy to mention here that, we will use the following notion to define Dirichlet boundary data in the fractional Sobolev space.
\begin{definition}
    We say function $u\leq 0$ on the boundary $\partial\Omega$ of a bounded domain $\Omega$, if $u=0$ in $\mathbb{R}^N\setminus\Omega$  and for any $\epsilon>0$, we have
    $$(u-\epsilon)^+\in W_0^{1,p}(\Omega).$$
In particular, $u=0$ on the boundary $\partial\Omega$ if $u\leq 0$ on $\partial\Omega$ and $u\geq 0$ on $\partial\Omega$.
\end{definition}
 \noindent We have the Poincar\'e inequality for all bounded domain $\Omega$,
  \begin{align}\label{lmn2.1 PQ}
        \|u\|_{L^p(\Omega)}\leq C\|\nabla u\|_{L^p(\Omega)}\,\,\, \forall\, u \in W_0^{1,p}(\Omega),
        \end{align}
        for some $C>0$. Thus, the norm \eqref{n1-1} reduces to the following norm in $W^{1,p}_0(\Omega)$,
        \begin{equation}\label{nh1}
            \|u\|_{W^{1,p}_0(\Omega)}=\left(\int_\Omega |\nabla u|^p\right)^\frac{1}{p}.
        \end{equation}
 Similarly, the Gagliardo seminorm \eqref{n g} becomes a norm on $W^{s,p}_0(\Omega)$, that is 
 \begin{equation}\label{nhs}
               \|u\|_{W_0^{s,p}(\Omega)}=\bigg(\int_\Omega \int_\Omega \frac{|u(x)-u(y)|^p}{|x-y|^{N+ps}}\d x \d y\bigg)^\frac{1}{p}.
           \end{equation}

 \begin{lemma}\label{sob eql} \cite{EVANS2022}
   For every $u \in W^{1,p}(\mathbb{R}^N)$ with $1<p<N$, we have 
   \begin{equation}\label{eq2.9}
           \bigg(\int_{\mathbb{R}^N}|u(x)|^{p^*}dx\bigg)^\frac{1}{p^*}\leq C\bigg(\int_{\mathbb{R}^N}|\nabla u(x)|^{p}dx\bigg)^\frac{1}{p},
       \end{equation}
       where $C>0$ is the best embedding constant and $p^*=\frac{Np}{N-p}$. In particular, when $\Omega$ is a bounded domain with Lipschitz boundary, then we have the embedding $W^{1,p}(\Omega) \hookrightarrow L^q(\Omega)$, which is compact for all $q\in[1,p^*]$ and is compact for $1\leq q<p^*$.  The result is also true for the space $W_0^{1,p}(\Omega)$.
   \end{lemma}
   \begin{lemma}\cite[Theorem 6.5]{NPV2012}
   For every $u \in W^{s,p}(\mathbb{R}^N)$ with $1<p<\frac{N}{s}$, we have 
   \begin{equation}
           \bigg(\int_{\mathbb{R}^N}|u(x)|^{p_s^*}dx\bigg)^\frac{1}{p_s^*}\leq C\bigg(\int_{\mathbb{R}^N} \int_{\mathbb{R}^N} \frac{|u(x)-u(y)|^p}{|x-y|^{N+sp}}\d x\d y \bigg)^\frac{1}{p},
       \end{equation}
        where $C>0$ is the best embedding constant and $p_s^*=\frac{Np}{N-ps}$. In particular, when $\Omega$ is a bounded domain with Lipschitz boundary, then we have the embedding $W^{s,p}(\Omega) \hookrightarrow L^q(\Omega)$, which is compact for all $q\in[1,p_s^*]$ and is compact for $1\leq q<p_s^*$. The result is also true for the space $W_0^{s,p}(\Omega)$. 
   \end{lemma}
\noindent   The next lemma is crucial for studying our problem.
   \begin{lemma}{\cite[Proposition 2.2]{NPV2012}, \cite[Lemma 2.1]{BSM2022}}
   Let $\Omega\subset\mathbb{R}^N$ be a bounded domain with Lipschitz boundary $\partial\Omega$. Then for $0<s<1 \leq p< \infty$
       there exists $C=C(N,p,s)>0$ such that 
       \begin{equation}\label{eq2.5IM}
       \begin{aligned}
    \|u\|_{W^{s,p}(\Omega)} \leq C\|u\|_{W^{1,p}(\Omega)}, \,\,\, \forall\, u\in W^{1,p}(\Omega).
     \end{aligned}
        \end{equation}
Moreover, for every $u \in W_0^{1,p}(\Omega)$ with $u=0$ in $\mathbb R^N\setminus\Omega$, we have
        \begin{equation}\label{sg-emb}
           \int_{\mathbb R^N} \int_{\mathbb R^N} \frac{|u(x)-u(y)|^p}{|x-y|^{N+sp}} \d x \d y \leq C \int_\Omega |\nabla u|^p \d x 
       \end{equation}
       and in particular, 
        \begin{equation}\label{sg-emb2}
           \int_{\Omega} \int_{\Omega} \frac{|u(x)-u(y)|^p}{|x-y|^{N+sp}} \d x \d y \leq C \int_\Omega |\nabla u|^p \d x \,\,\, \forall \, u \in W_0^{1,p}(\Omega).
           %~\text{with}~u=0~\text{in}~\mathbb R^N\setminus\Omega.
       \end{equation}
       \end{lemma}
\noindent We now define the solution space for our problem.    
\begin{definition}
  Let $\Omega\subset\mathbb{R}^N$ be a bounded domain with Lipschitz boundary $\partial \Omega$ and let $0<s<1\leq p<\infty$. We define the space $X_0^{s,p}(\Omega)$ as the closure of $C_c^\infty(\Omega)$ with respect to the following norm:
  \begin{align}\label{eq2.7 N}
      \|u\|_{X_0^{s,p}(\Omega)}&=\bigg(\int_\Omega|\nabla u|^p \d x+\int_{\Omega} \int_{\Omega} \frac{|u(x)-u(y)|^p}{|x-y|^{N+sp}}\d x\d y \bigg)^\frac{1}{p},\,\,\,\,\forall\, u\in C_c^\infty(\Omega).
      \end{align} 
      \end{definition}
      \begin{remark}\label{rm 2.1}
      On using the inequalities \ref{sg-emb}, \eqref{sg-emb2} and the Poincar\'e inequality in \eqref{lmn2.1 PQ}, we obtain the following equivalent norms in $X_0^{s,p}(\Omega)$:
      \begin{align*}
          \|u\|_{X_0^{s,p}(\Omega)}&=\bigg(\int_\Omega|\nabla u|^p \d x \bigg)^\frac{1}{p},\,\,\,\forall \, u \in X_0^{s,p}(\Omega) \text{ and }\\
          \|u\|_{X_0^{s,p}(\Omega)}&=\bigg(\int_\Omega|\nabla u|^p \d x+\int_{\mathbb{R}^N} \int_{\mathbb{R}^N} \frac{|u(x)-u(y)|^p}{|x-y|^{N+sp}}\d x\d y \bigg)^\frac{1}{p},\,\,\,\forall \, u \in X_0^{s,p}(\Omega).
      \end{align*}
Thus, the space $X_0^{s,p}(\Omega)$ is characterized as
    \begin{equation}\nonumber 
          X_0^{s,p} (\Omega)={\{u\in W^{1,p} (\Omega): u=0 \text{ in } \mathbb R^N \setminus \Omega} \text{ and } [u]_{W^{s,p}(\Omega)}< + \infty \}.
          \end{equation}
    \end{remark}
    \noindent Note that on using \eqref{eq2.9}, we get the following Sobolev inequality in $X_0^{s,p} (\mathbb{R}^N)$. 
     \begin{equation}\label{eq 2.15}
           \bigg(\int_{\mathbb{R}^N}|u(x)|^{p^*}dx\bigg)^\frac{1}{p^*}\leq C\bigg(\int_{\mathbb{R}^N}|\nabla u(x)|^{p}dx +\int_{\mathbb{R}^N} \int_{\mathbb{R}^N} \frac{|u(x)-u(y)|^p}{|x-y|^{N+sp}}\d x\d y\bigg)^\frac{1}{p},
       \end{equation}
       where $C>0$ is the best embedding constant and $p^*=\frac{Np}{N-p}>\frac{Np}{N-ps}:=p_s^*$. Therefore, the inequality \eqref{eq2.9} and Remark \ref{rm 2.1} assert that
          \begin{equation*}
            \|u\|_{L^p{^*}(\Omega)}=\|u\|_{L^p{^*}(\mathbb R^N)} \leq C\|\nabla u\|_{L^p(\mathbb R^N) }\leq C\|u\|_{X_0^{s,p}(\Omega)},\ \,\,\, \forall\, u \in X_0^{s,p}(\Omega).
             \end{equation*} 
    \begin{theorem}\label{thm cpt}
     Let $0<s<1\leq p<\infty$ and let $\Omega\subset\mathbb{R}^N$ be a bounded domain with Lipschitz boundary $\partial \Omega$. Then we have
             \begin{equation}\label{eq2.8S}
                  \|u\|_{L^p{^*}(\Omega)} \leq C\|u\|_{X_0^{s,p}(\Omega)},\ \,\,\, \forall\, u \in X_0^{s,p}(\Omega).
             \end{equation}
             Moreover, the embedding $X_0^{s,p}(\Omega)\hookrightarrow L^q(\Omega)$ is continuous for $1\leq q\leq p^*$ and is compact $1\leq q< p^*$.
    \end{theorem}
  \noindent The following theorem characterises the space $X^{s,p}_0(\Omega)$.
            \begin{theorem}\label{thm prop}
                Let $0<s<1\leq p<\infty$ and let $\Omega\subset\mathbb{R}^N$ be a bounded domain with Lipschitz boundary $\partial \Omega$. Then for all $p\in[1,\infty)$, the space $X^{s,p}_0(\Omega)$ is a Banach space endowed with the norm \eqref{eq2.7 N}. It is separable for all $p\in[1,\infty)$ and is reflexive for all $p\in(1,\infty)$. In particular, when $p=2$, the space $X_0^{s,2}(\Omega)$ reduces to a Hilbert space with respect to the inner product, 
                \begin{equation}
                    \langle u,v\rangle_{X_0^{s,2}(\Omega)}=\int_{\Omega} \nabla u\cdot\nabla v \d x+\int_{\Omega} \int_{\Omega} \frac{(u(x)-u(y))(v(x)-v(y))}{|x-y|^{N+2s}} \d x \d y,
                \end{equation}
                where ``$\cdot$" denotes the standard scalar product in $\mathbb{R}^N$.
            \end{theorem}
            \begin{proof}
                By definition, $X^{s,p}_0(\Omega)$ is a Banach space. For any $u\in X^{s,p}_0(\Omega)$, choose $A_u(x)=\nabla u(x)$ and $B_u(x,y)=\frac{u(x)-u(y)}{|x-y|^{\frac{N}{p}+s}}$ and define the map $T: X^{s,p}_0(\Omega) \rightarrow L^p(\Omega)\cross L^p(\Omega \cross \Omega)$ such that
                $$T(u)=(A_u,B_u).$$
             Since, $\|Tu\|_{L^p(\Omega)\cross L^p(\Omega \cross \Omega)}=\|u\|_{X^{s,p}_0(\Omega)}$  $\forall\,u\in X^{s,p}_0(\Omega)$, we obtain $T$ is an isometry into the closed subspace of $L^p(\Omega)\cross L^p(\Omega \cross \Omega))$. Thus we get $X^{s,p}_0(\Omega)$ is is reflexive for all $p\in(1,\infty)$ and is separable for all $p\in[1,\infty)$. Finally, using the fact that $W_0^{1,2}(\Omega)$ and $W_0^{s,2}(\Omega)$ are Hilbert spaces, we conclude $X^{s,2}_0(\Omega)$ is a Hilbert space.
            \end{proof}
\begin{remark}
It is noteworthy to mention here that for $p=2$ (\cite[Proposition 3.6]{NPV2012}),
$$[u]_{W_0^{s,2}(\Omega)}=\left(\int_\Omega \int_\Omega \frac{|u(x)-u(y)|^2}{|x-y|^{N+2s}}\d x \d y\right)^\frac{1}{2}=\left(\int_\Omega |(-\Delta)^\frac{s}{2} u|^2 \d x\right)^\frac{1}{2},$$
up to a constant. The following equality holds up to a constant factor (see \cite{NPV2012}). We also refer to \cite{K2017} for different notions of fractional Laplacian and their equivalence.
$$\int_\Omega \int_\Omega \frac{(u(x)-u(y))(\phi(x)-\phi(y))}{|x-y|^{N+2s}}\d x \d y=\int_\Omega (-\Delta)^\frac{s}{2} u \cdot(-\Delta)^\frac{s}{2} \phi\d x,~\forall\,\phi\in W_0^{s,2}(\Omega).$$
\end{remark}
\noindent We now recall the Definitions of Marcinkiewicz space and the H\"older space. The Marcinkiewicz space (or weak $L^q(\Omega)$) is denoted by $M^q(\Omega)$ and is defined for every $0<q<\infty$ as the space of all measurable functions $f:\Omega \rightarrow \mathbb R$ such that the corresponding distribution function $\phi_f(t)=m({\{x\in\Omega: |f(x)|>t}\})$ satisfies the following estimate:
\begin{equation*}
\phi_f(t)=m({\{x\in\Omega: |f(x)|>t}\})\leq \frac{C}{t^q} \,\,\,\,\,\, \forall \,t\geq t_0,
\end{equation*}
where $t_0>0$, $C>0$ and $m$ refers the Lebesgue measure. When $\Omega$ is bounded, we have $M^q(\Omega) \subset M^p(\Omega),\,\,\,\forall\, q\geq p$. Furthermore, the following continuous embedding holds. 
\begin{equation}\label{embed2.1}
L^q(\Omega)\hookrightarrow M^q(\Omega) \hookrightarrow L^{q-\epsilon}(\Omega),
 \end{equation}
 for every $1<q<\infty$ and for all $0<\epsilon<q-1$. For $0<\beta\leq 1$ and $k\in\mathbb{N}\cup\{0\}$, the H\"older space $C^{k,\beta}(\bar{\Omega})$, \cite{EVANS2022} consists of all functions $u\in C^k(\bar\Omega)$ such that $\|u\|_ {C^{k,\beta}(\bar\Omega)}< \infty,$ where 
 $$\|u\|_{C^{k,\beta}(\bar\Omega)}= \sum_{|\alpha|\leq k} \sup_{x\in \bar \Omega} |D^\alpha u(x)|+\sum_{|\alpha|=k}\sup_{x,y\in \Omega,x\neq y} {\left\{ \frac{|D^\alpha u(x)-D^\alpha u(y)|}{|x-y|^\beta}\right\}}.$$
 We denote $\mathcal{M}(\Omega)$ as the space of all finite Radon measures on $\Omega$ and for $\mu \in \mathcal{M}(\Omega)$, we define the total variation norm of $\mu$ by $\|\mu\|_{\mathcal{M}(\Omega)}=\int_\Omega \d|\mu|$. 
\begin{definition}\label{def2.1 M}
We say a sequence $(\mu_n)$ in $\mathcal{M}(\Omega) $ converges to $\mu\in \mathcal{M}(\Omega)$, i.e. $\mu_n \rightharpoonup \mu$ in the sense of measure (\cite{F1999}) if for all $\phi \in C_c^\infty(\Omega)$, we have
\begin{center}
    $\int_\Omega \phi d \mu_n \rightarrow \int_\Omega \phi d\mu.$
    \end{center}
   \end{definition}
 \noindent  For is an open bounded subset $\Omega\subset\mathbb{R}^N$ with boundary $\partial\Omega$, we define, $\delta(x):=dist(x,\partial \Omega)$.
We conclude this section with the Definition of the following two truncation functions. For $k>0$, we define
\begin{equation*}
\begin{aligned}
    T_k(x)=&\max{\{-k, \min{\{k,x}\}\}} \text{ and } G_k(x)=(|x|-k)^+ sign(x), ~ \forall \,x \in \mathbb{R}.
\end{aligned}
\end{equation*}\\
Observe that $T_k(x)+G_k(x)=x$ for all $x \in\mathbb{R}$.

\section{Existence of weak solution to the problem \ref{eq2:problem}}\label{sec3}
In this section, we establish the existence of weak solutions to the problem \eqref{eq2:problem}. We begin with the definition of weak solutions to \eqref{eq2:problem}.
\begin{definition}\label{d weak}
For $0<\gamma \leq 1$, we say $u\in X_0^{s,1}(\Omega)$ is a weak solution of $\eqref{eq2:problem}$ if 
\begin{equation}\label{eq2.8:rel K}
    \text{ for every } K \Subset \Omega, \text{ there exists } C_K>0 \text{ such that } u\geq C_K >0 \end{equation} and 
    \begin{equation}\label{eq2.9:WF}
    \int_\Omega \nabla u\cdot\nabla \phi+\int_\Omega (-\Delta)^\frac{s}{2} u\cdot(-\Delta)^\frac{s}{2} \phi=\int_\Omega \frac{f \phi}{u^\gamma} + \int_\Omega \phi d\mu\,\,\,\,\,\, \forall ~\phi \in C_c^\infty(\Omega).
\end{equation}
For $\gamma>1$, we say $u\in X^{s,1}_{loc}(\Omega)$ is a weak solution to \eqref{eq2:problem} if $u$ satisfy \eqref{eq2.8:rel K} and \eqref{eq2.9:WF} with $T_k ^\frac{\gamma+1}{2}(u) \in X_0^{s,2}(\Omega) $ for every $k>0$.
\end{definition}
\noindent Due to the presence of a Radon measure in the weak formulation (refer Definition \ref{d weak}), we proceed with the weak convergence method. Let us consider the following sequence of problems.
\begin{equation}\label{eq2.10:SWP}
\begin{aligned}
    -\Delta u_n +(-\Delta)^s & u_n=\frac{f_n}{(u_n + \frac{1}{n})^{\gamma}}+\mu_n \text{ in } \Omega,\\
     & u_n =0 \text{ in } \mathbb R^N \setminus \Omega,
   \end{aligned}
    \end{equation}
    where $(\mu_n)$ is a sequence of nonnegative, bounded, smooth functions such that $\mu_n\rightharpoonup\mu$ in the sense of Definition \ref{def2.1 M} and $f_n$ is the truncation at level $n$ of $f$. We now consider the weak formulation of the equation \eqref{eq2.10:SWP}.
    \begin{equation}\label{eq2.11:SWF}
        \int_\Omega \nabla u_n\cdot\nabla \phi+\int_\Omega (-\Delta)^\frac{s}{2} u_n\cdot(-\Delta)^\frac{s}{2} \phi=\int_\Omega \frac{f_n \phi}{(u_n +\frac{1}{n})^\gamma} + \int_\Omega \phi d\mu_n\,\,\,\,\,\, \forall~ \phi \in C_c^\infty(\Omega).
    \end{equation}
We first prove the existence of a solution to \eqref{eq2.10:SWP}. We begin with the following lemma.    
    \begin{lemma} \label{lmm2.5}
        Problem \eqref{eq2.10:SWP} admits a nonnegative weak solution $u_n \in X_0 ^{s,2}(\Omega) \cap L^{\infty}(\Omega).$
    \end{lemma}
    \begin{proof}
     To prove the existence results, we use Schauder's fixed point theorem. For each $n\in \mathbb N$, we define 
    \begin{equation}\nonumber
    G:L^2(\Omega) \rightarrow L^2(\Omega),
    \end{equation} by $G(v)=w$. The map $G$ is well-defined. Indeed, from the Lax-Milgram lemma, we deduce that for every $v\in L^2(\Omega)$ there exists a unique weak solution $w\in X^{s,2}_0(\Omega)$ to the following problem:
    \begin{equation}\label{eq2.12:map}
\begin{aligned}
    -\Delta w +(-\Delta)^s & w=\frac{f_n}{(|v| + \frac{1}{n})^{\gamma}}+\mu_n\,\,  \text{in} \,\,\Omega,\\
     & w =0 \,\text{in} \, \mathbb R^N \setminus \Omega.
   \end{aligned}
    \end{equation}
    We remark that $X^{s,2}_0(\Omega)$ can be taken as the test functions space in the weak formulation \eqref{eq2.11:SWF}, since, $C_c^\infty(\Omega)$ is dense in $L^2(\Omega)$. To apply Schauder's fixed point theorem, we first prove that $G$ is bounded over $L^2(\Omega).$ Recall the following Poincar\'e inequalities,
    \begin{equation}
        \lambda_1 \int_\Omega |w|^2 \leq \int_\Omega |\nabla w|^2 \label{eq2.13:P L}
            \end{equation}~ \text{and}
            \begin{equation}
            \lambda_{1,s}\int_\Omega |w|^2 \leq \int_\Omega |(-\Delta)^\frac{s}{2}w|^2, \label{eq2.14:P FL}
             \end{equation}
             where $\lambda_1$ and $\lambda_{1,s}$ are the principle eigenvalue corresponding to the operators $(-\Delta)$ and $(-\Delta)^s$ respectively. Adding the inequalities \eqref{eq2.13:P L}, \eqref{eq2.14:P FL} and then using H\"older's inequality, we obtain
\begin{align}\nonumber
    (\lambda_1 +\lambda_{1,s})\int_\Omega |w|^2 & \leq \int_\Omega |\nabla w|^2 + \int_\Omega |(-\Delta)^\frac{s}{2}w|^2\\
   & =\int_\Omega \nabla w \cdot \nabla w + \int_\Omega (-\Delta)^\frac{s}{2}w \cdot (-\Delta)^\frac{s}{2}w \nonumber \\
   & =\int_\Omega \frac{f_n w}{(|v| +\frac{1}{n})^\gamma} + \int_\Omega \mu_n w \nonumber\\
   & \leq \int_\Omega n.n^\gamma |w| +C(n)\int_\Omega |w|\nonumber\\
   &= (n^{\gamma +1} +C(n)) \int _\Omega |w| \nonumber\\
   &\leq C(n^{\gamma +1} +C(n))\left(\int _\Omega |w|^2\right)^\frac{1}{2}.\label{eq2.15 inq}\,\,\,\,\
   \end{align}
    Therefore, we obtain
    \begin{equation}\label{eq2.16:G bd}
        \|w\|_{L^2(\Omega)} \leq \frac{C(n^{\gamma +1} +C(n))}{(\lambda_1 +\lambda_{1,s})}=C(n,\gamma), 
    \end{equation}
where $C(n,\gamma)$ is independent of $v$. We now prove the continuity of $G$ in $L^2(\Omega)$. Let $(v_k)$ be a sequence converging to $v$ in $L^2(\Omega)$. From the Lebesgue dominated convergence theorem, we have
\begin{equation}\nonumber
    \Bigg| \Bigg|\left(\frac{f_n}{(|v_k|+\frac{1}{n})^\gamma}+\mu_n\right)-\left(\frac{f_n}{(|v|+\frac{1}{n})^\gamma}+\mu_n\right) \Bigg| \Bigg|_{L^2(\Omega)} \longrightarrow 0 \text{ as }  k \rightarrow \infty,
\end{equation}
Therefore, using the uniqueness of weak solution, we obtain that  $G(v_k):=w_k$ converges to $G(v):=w$. Thus $G$ is continuous in $L^2(\Omega)$. Now inequalities \eqref{eq2.15 inq} and \eqref{eq2.16:G bd} give
\begin{align}\nonumber
    \int_\Omega |\nabla w|^2 + \int_\Omega |(-\Delta)^\frac{s}{2}w|^2 & \leq C(n^{\gamma +1}+C(n))\left(\int_\Omega |w|^2\right)^\frac{1}{2}\\
    &\leq C(n^{\gamma +1}+C(n))C(n,\gamma)\nonumber \\ 
    & =C(n,\gamma).\nonumber
    \end{align}
   Therefore, we have
   \begin{align} \nonumber
    \int_\Omega |\nabla G(v)|^2 + \int_\Omega |(-\Delta)^\frac{s}{2}G(v)|^2 =\int_\Omega |\nabla w|^2 + \int_\Omega |(-\Delta)^\frac{s}{2}w|^2 & \leq C(n,\gamma).
\end{align}
Thus, we get $$\|w\|_{X^{s,2}_0(\Omega)} =\|G(v)\|_{X^{s,2}_0(\Omega)} \leq  C(n,\gamma), \text{ for any  }v \in L^2(\Omega).$$
Therefore, using the compact embedding lemma, we conclude that $G(L^2(\Omega))$ is relatively compact in $ L^2(\Omega)$. Hence, we get that for each $n\in\mathbb{N}$, the map $G$ has a fixed point $u_n \in L^2(\Omega)$ which is a weak solution to \eqref{eq2.10:SWP} in $X_0^{s,2}(\Omega)$. The boundedness of $u_n$ follows from \cite{S1995}. Finally, since, $\left (\frac{f_n}{(|u_n| + \frac{1}{n})^{\gamma}}+\mu_n \right ) \geq 0$ then from the maximum principle \cite[Theorem 1.2]{BDVV2022} for $(-\Delta +(-\Delta)^s)$, we deduce that $u_n \geq 0$. This completes the proof.
\end{proof}
We now prove the boundedness of $(u_n)$ on relatively compact subsets of $\Omega$.
\begin{lemma}\label{lmm2.6}
    The sequence $(u_n)$ is such that  for every $K \Subset \Omega$ there exists $C_K$ (independent of $n$) such that $u_n \geq C_K > 0$, for {\it{a.e.}} in $K$, and for every $n$, where $ n \in \mathbb N. $
\end{lemma}
\begin{proof}
  Consider the following sequence of problems:
\begin{equation}
    \begin{aligned}\label{eq2.17 SSP}
    -\Delta w_n +(-\Delta)^s & w_n=\frac{f_n}{(w_n + \frac{1}{n})^{\gamma}}\,\,  \text{in} \,\,\Omega,\\
     & w_n =0 \,\,\,\, \,\text{in} \, \mathbb R^N \setminus \Omega.
   \end{aligned}
   \end{equation}
 We show that for every weak solution $w_n$ of \eqref{eq2.17 SSP} and for all $K \Subset \Omega$, there exists a constant $C_K>0$, (independent of $n$), such that $w_n \geq C_K >0$ {\it{a.e.}} in $K$. Suppose $w_n$ and $w_{n+1}$ are two weak solutions of \eqref{eq2.17 SSP}, respectively. Then, we have 
\begin{align}
    \int_\Omega \nabla w_n\cdot \nabla\phi +\int_\Omega (-\Delta)^\frac{s}{2}w_n\cdot (-\Delta)^\frac{s}{2}\phi  &=  \int_\Omega f_{n}\bigg(w_n +\frac{1}{n}\bigg)^{-\gamma}\phi\label{eq2.18 NSS}\\
    &\leq \int_\Omega f_{n+1}\bigg(w_{n} +\frac{1}{n+1}\bigg)^{-\gamma}\phi.\label{eq2.19 NSSP}
    \end{align}
 and 
 \begin{equation}
 \begin{aligned}\label{eq2.20}
    \int_\Omega \nabla w_{n+1}\cdot \nabla\phi +\int_\Omega (-\Delta)^\frac{s}{2}w_{n+1}\cdot (-\Delta)^\frac{s}{2}\phi  &=  \int_\Omega f_{n+1}\bigg(w_{n+1} +\frac{1}{n+1}\bigg)^{-\gamma}\phi.
    \end{aligned}
 \end{equation}
Subtracting \eqref{eq2.20} from \eqref{eq2.19 NSSP} and putting $\phi=(w_n - w_{n+1})^+$ as test function, we deduce
\begin{align*}
 0 & \leq \int_\Omega \nabla(w_n-w_{n+1})\cdot \nabla(w_n-w_{n+1})^+ +\int_\Omega (-\Delta)^\frac{s}{2}(w_n-w_{n+1})\cdot (-\Delta)^\frac{s}{2}(w_n-w_{n+1})^+ \\ 
 & =\int_\Omega |\nabla(w_n-w_{n+1})^+|^2 + \int_\Omega |(-\Delta)^\frac{s}{2}(w_n-w_{n+1})^+|^2 \\ 
 & \leq  \int_\Omega f_{n+1}\bigg[\bigg(w_n +\frac{1}{n+1}\bigg)^{-\gamma}-\bigg(w_{n+1} +\frac{1}{n+1}\bigg)^{-\gamma}\bigg](w_n-w_{n+1})^+ \leq 0.
 \end{align*}
 Therefore, $\|(w_n-w_{n+1})^+\|_{X^{s,2}_0(\Omega)}=0$, which implies $(w_n-w_{n+1})^+=0$ {\it{a.e.}} in $\Omega$. Hence, we conclude that the sequence $(w_n)$ is monotonically increasing, $w_{n+1}\geq w_n~\forall~ n\in\mathbb{N}$. In particular, $w_n\geq w_1$, where $w_1\in L^\infty(\Omega)$ (refer to Garain and Ukhlov \cite{GU2022}) is a weak solution to the following problem
 \begin{equation}\label{eq2.21}
     -\Delta w_1 + (-\Delta)^sw_1=\frac{f_1}{(w_1+1)^\gamma}.
 \end{equation}
 Clearly,  $\frac{f_1}{(w_1+1)^\gamma} \in L^\infty(\Omega)$ by the definition of $f_1$. Let $\|w_1\|_{L^\infty(\Omega)} \leq M$. Then we derive (in a weak sense) 
 \begin{align*}
     -\Delta w_1 + (-\Delta)^sw_1 =\frac{f_1}{(w_1+1)^\gamma} & \geq \frac{f_1}{(1+\|w_1\|_{L^\infty(\Omega)})^\gamma}\geq \frac{f_1}{(1+M)^\gamma} >0.
 \end{align*}
 Since $\frac{f_1}{(1+M)^\gamma}\not\equiv 0$ in $K$, we get $w_1 > 0$ by using the strong maximum principle over $(-\Delta+((-\Delta)^s))$\cite{PFR2019}. Since $K\Subset\Omega$, then there exists a constant $M_K>0$ such that $w_1 \geq M_K > 0$. Thus using the monotinicity of $(w_n)$ we get $w_n \geq w_1 \geq M_K > 0$ for all $n\in\mathbb{N}$. 

We conclude the proof with the fact that $u_n \geq w_n$ {\it{a.e.}} in $K$. 
Indeed, subtracting \eqref{eq2.18 NSS} from \eqref{eq2.11:SWF} and taking $\phi=(u_n-w_n)^-$ as the test function, we get
\begin{align*}
    0&\geq -\int_\Omega \|\nabla(u_n-w_{n})^-\|^2 -\int_\Omega \|(-\Delta)^\frac{s}{2}(u_n-w_{n})^-\|^2\\ 
    &=\int_\Omega  \nabla(u_n-w_n)\cdot \nabla(u_n-w_n)^- +\int_\Omega (-\Delta)^\frac{s}{2}(u_n-w_n) \cdot (-\Delta)^\frac{s}{2}(u_n-w_n)^{-}\\
     &=\int_\Omega f_{n}\bigg[\bigg(u_n +\frac{1}{n}\bigg)^{-\gamma}-\bigg(w_{n} +\frac{1}{n}\bigg)^{-\gamma}\bigg](u_n-w_{n})^- +\int_\Omega \mu_n (u_n-w_n)^{-}\\ 
     &\geq \int_\Omega f_{n}\bigg[\bigg(u_n +\frac{1}{n}\bigg)^{-\gamma}-\bigg(w_{n} +\frac{1}{n}\bigg)^{-\gamma}\bigg](u_n-w_{n})^-  \geq 0.
\end{align*}
 Thus, we obtain $u_n\geq w_n$ {\it{a.e.}} in $\Omega$. Therefore, for every $K\Subset\Omega$, there exists $C_K>0$ such that $u_n\geq C_K>0$ {\it{a.e.}} in $K$. This completes the proof.
 \end{proof}
\noindent We now prove the existence of a weak solution to the problem \eqref{eq2:problem}. We divide the proof in two cases with $0<\gamma\leq 1$ and $\gamma>1$.
\subsection{Existence result when \texorpdfstring{$0<\gamma\leq1$}{0<γ≤1}}\text{ }\\ 
We establish the existence results by using the embedding of the Marcinkiewicz space \eqref{embed2.1} together with the following lemma. 
 \begin{lemma}\label{lmn3.1}
     For each $n\in\mathbb{N}$, let $u_n$ be a weak solution to the problem \eqref{eq2.10:SWP} for $0<\gamma \leq 1$. Then $u_n$ is bounded in $X_0^{s,q}(\Omega)$ for every $q<\frac{N}{N-1}$.
 \end{lemma}
 \begin{proof}
 We first claim that, $$\int_\Omega|\nabla T_k(u_n)|^2 +\int_\Omega |(-\Delta)^\frac{s}{2}T_k(u_n)|^2 \leq Ck,~\forall \, k\geq 1, $$  where $C>0$ is a fixed constant. Let us choose $\phi=T_k(u_n)$ as a test function in the weak formulation \eqref{eq2.11:SWF}. Then we get
\begin{align}\label{eq3.1 BTF}
    \int_\Omega |\nabla T_k(u_n)|^2+\int_\Omega |(-\Delta)^\frac{s}{2} T_k(u_n)|^2 & =\int_\Omega \nabla u_n \cdot \nabla T_k(u_n)+\int_\Omega (-\Delta)^\frac{s}{2} u_n\cdot(-\Delta)^\frac{s}{2} T_k(u_n)\nonumber\\ 
    & = \int_\Omega \frac{f_n T_k(u_n)}{(u_n +\frac{1}{n})^\gamma} + \int_\Omega \mu_n T_k(u_n).
\end{align}
Now,            
\begin{align}\label{sg3.2}
    \int_\Omega \frac{f_n T_k(u_n)}{(u_n +\frac{1}{n})^\gamma} &=\int_{(u_n +\frac{1}{n})\leq k}\frac{f_n T_k(u_n)}{(u_n +\frac{1}{n})^\gamma} +\int_{(u_n +\frac{1}{n})\geq k} \frac{f_n T_k(u_n)}{(u_n +\frac{1}{n})^\gamma}\nonumber\\
    & \leq \int_{(u_n +\frac{1}{n})\leq k}\frac{f_n u_n}{(u_n +\frac{1}{n})^\gamma}+\int_{(u_n +\frac{1}{n})\geq k}\frac{f k}{k^\gamma} \nonumber\\
   &=\int_{(u_n +\frac{1}{n})\leq k}\frac{f_n u_n^\gamma u_n^{1-\gamma}}{(u_n + \frac{1}{n})^\gamma}+k^{1-\gamma} \int_{(u_n +\frac{1}{n})\geq k} f  \nonumber\\ 
   & \leq \int_{(u_n +\frac{1}{n})\leq k} f u_n^{1-\gamma}+k^{1-\gamma} \int_{(u_n +\frac{1}{n})\geq k} f  \nonumber\\
   & \leq \int_{(u_n +\frac{1}{n})\leq k} f k^{1-\gamma}+k^{1-\gamma} \int_{(u_n +\frac{1}{n})\geq k} f  \nonumber\\ 
   & =k^{1-\gamma} \int_\Omega f  \nonumber\\ 
   & = C k^{1-\gamma}  \nonumber\\
   &\leq Ck, \,\,\,\, \forall \,k \geq 1.
   \end{align}
   Also,
   \begin{align}\label{sg3.3}
    \int_\Omega \mu_n T_k(u_n) \leq k\|\mu_n\|_{L^1(\Omega)} \leq Ck.
    \end{align} 
Thus, using the inequalities \eqref{sg3.2} and \eqref{sg3.3} in \eqref{eq3.1 BTF}, we get
\begin{equation}\label{eq3.4 BTN}
\int_\Omega |\nabla T_k(u_n)|^2+\int_\Omega |(-\Delta)^\frac{s}{2} T_k(u_n)|^2 \leq Ck, \,\,\,\forall \, k \geq 1. 
\end{equation}
From the mixed Sobolev inequality \eqref{eq2.8S} with $p=2$, we obtain
\begin{align}\label{eq3.5 BSI}
 \frac{1}{C^2}\bigg(\int_\Omega |T_k(u_n)|^{2^*}\bigg)^\frac{2}{{2^*}} & \leq\int_\Omega |\nabla T_k(u_n)|^2+\int_\Omega |(-\Delta)^\frac{s}{2}T_k(u_n)|^2 \leq Ck, \,\,\,\, \forall \, k\geq 1. 
 \end{align}
We now split the proof in the line of obtaining the boundedness of $\nabla u_n$ and ($-\Delta)^\frac{s}{2}u_n$ in $M^\frac{N}{N-1}(\Omega),$ that is we prove that
\begin{align*}
    m(\{|\nabla u_n| \geq t\}) &\leq \frac{C}{t^\frac{N}{N-1}}~\text{ and }
    m(\{|(-\Delta)^\frac{s}{2} u_n| \geq t\}) \leq \frac{C}{t^\frac{N}{N-1}},~\forall~ t \geq 1,
\end{align*}
where $C$ is constant and $m$ refers to the Lebesgue measure. We use the following set of inclusions: 
\begin{align*}
{\{|\nabla u_n| \geq t}\} & ={\{|\nabla u_n| \geq t,u_n<k}\} \cup {\{|\nabla u_n| \geq t,u_n \geq k}\}\\ 
& \subset {\{|\nabla u_n| \geq t,u_n<k}\} \cup {\{u_n \geq k}\} \subset \Omega
\end{align*} 
and
\begin{align*}
{\{|(-\Delta)^\frac{s}{2} u_n| \geq t}\} & ={\{|(-\Delta)^\frac{s}{2} u_n| \geq t,u_n<k}\} \cup {\{|(-\Delta)^\frac{s}{2} u_n| \geq t,u_n \geq k}\}\\ 
& \subset {\{|(-\Delta)^\frac{s}{2} u_n| \geq t,u_n<k}\} \cup {\{u_n \geq k}\} \subset \Omega.
\end{align*}
Thus, by using the monotone and sub-additive properties of Lebesgue measure $m$, we obtain
\begin{align}\label{eq3.6 MI}
    m({\{|\nabla u_n| \geq t}\}) \leq m({\{|\nabla u_n| \geq t,u_n<k}\}) + m({\{u_n \geq k}\})
\end{align}
and 
\begin{align}\label{eq3.7 MI}
    m({\{|(-\Delta)^\frac{s}{2} u_n| \geq t}\}) \leq m({\{|(-\Delta)^\frac{s}{2} u_n| \geq t,u_n<k}\}) + m({\{u_n \geq k}\}).
\end{align}
We now estimate the terms in the above two inequalities \eqref{eq3.6 MI} and \eqref{eq3.7 MI}. Let
\begin{align*}
    I_1&=\{x \in \Omega:u_n \geq k\},\\
    I_2&=\{|\nabla u_n| \geq t,u_n<k\}~\text{ and }\\
    I_3&=\{|(-\Delta)^\frac{s}{2}u_n| \geq t,u_n<k\}.
\end{align*}
$\bullet$ {\bf{Estimate on $I_1$.}} Recall the Definition of $T_k$. Observe that $T_k(u_n)=k$ on $I_1$. Therefore, using \eqref{eq3.5 BSI}, we get
 \begin{align*}
     C^2Ck \geq \bigg(\int_{I_1}|T_k(u_n)|^{2^*}\bigg)^\frac{2}{2^*}=\bigg(\int_{I_1} k^{2^*}\bigg)^\frac{2}{2^*}=k^2 m(I_1)^\frac{2}{2^*}=k^2 m(\{u_n\geq k\})^\frac{2}{2^*},~\forall\, k\geq 1.
 \end{align*}
 Thus, we have
 \begin{align}\label{eq3.8A}
    m(\{u_n\geq k\}) & \leq \frac{C}{k^\frac{2^*}{2}} =\frac{C}{k^\frac{N}{N-2}},\,\,\,\forall \, k \geq 1.
 \end{align}
Therefore, ($u_n$) is bounded in the Marcinkiewicz space $M^\frac{N}{N-2}(\Omega)$.\\
$\bullet$ {\bf{Estimate on $I_2$.}} On using \eqref{eq3.4 BTN}, we get 
 \begin{align}\label{eq3.9B}
     m\bigg(\{|\nabla u_n| \geq t,u_n<k\}\bigg) &\leq \frac{1}{t^2} \int_{I_2}|\nabla u_n|^2 =\frac{1}{t^2} \int_{I_2}|\nabla T_k(u_n)|^2\nonumber\\ 
      & \leq \frac{1}{t^2} \bigg[ \int_{\Omega}|\nabla T_k(u_n)|^2+\int_{\Omega}|(-\Delta)^\frac{s}{2}T_k(u_n)|^2\bigg]\nonumber\\
      &\leq \frac{Ck}{t^2}, \,\,\,\forall\, k \geq 1.
      \end{align}
$\bullet$ {\bf{Estimate on $I_3$.}}
      Similar to the above estimate, using \eqref{eq3.4 BTN}, we have 
 \begin{align}\label{eq3.10C}
     m\bigg(\{|(-\Delta)^\frac{s}{2}u_n| \geq t,u_n<k\}\bigg) &\leq \frac{1}{t^2} \int_{I_3}|(-\Delta)^\frac{s}{2}u_n|^2=\frac{1}{t^2} \int_{I_3}|(-\Delta)^\frac{s}{2}T_k(u_n)|^2\nonumber\\ 
      & \leq \frac{1}{t^2} \bigg[ \int_{\Omega}|\nabla T_k(u_n)|^2+\int_{\Omega}|(-\Delta)^\frac{s}{2}T_k(u_n)|^2\bigg]\nonumber\\
      &\leq \frac{Ck}{t^2}, \,\,\,\forall\, k \geq 1.
      \end{align}    
 Now, combining the inequalities \eqref{eq3.8A}, \eqref{eq3.9B} with \eqref{eq3.6 MI}, we deduce
 \begin{align}\label{eq3.11D}
    m({\{|\nabla u_n| \geq t}\}) &\leq m({\{|\nabla u_n| \geq t,u_n<k}\}) + m({\{u_n \geq k}\})\nonumber\\ 
    & \leq \frac{Ck}{t^2}+ \frac{C}{k^\frac{N}{N-2}},\,\,\,\forall~ k\, \geq 1.
     \end{align}
Again the inequalities \eqref{eq3.8A}, \eqref{eq3.10C} combined with \eqref{eq3.7 MI}, give
 \begin{align}\label{eq3.12E}
    m({\{|(-\Delta)^\frac{s}{2} u_n| \geq t}\}) &\leq m({\{|(-\Delta)^\frac{s}{2} u_n| \geq t,u_n<k}\}) + m({\{u_n \geq k}\})\nonumber\\ 
    & \leq \frac{Ck}{t^2}+ \frac{C}{k^\frac{N}{N-2}}, \,\,\,\forall~ k\, \geq 1.
     \end{align}
Choosing $k=t^\frac{N-2}{N-1}$ in \eqref{eq3.11D} and \eqref{eq3.12E}, we get
   \begin{align}
       m({\{|\nabla u_n| \geq t}\})&\leq \frac{Ct^\frac{N-2}{N-1}}{t^2}+ \frac{C}{t^\frac{N}{N-1}} \nonumber\\
       & \leq \frac{C}{t^\frac{N}{N-1}},\,\,\,\forall\, t \geq 1
       \end{align}
       and 
        \begin{align}
       m({\{|(-\Delta)^\frac{s}{2} u_n| \geq t}\})&\leq \frac{Ct^\frac{N-2}{N-1}}{t^2}+ \frac{C}{t^\frac{N}{N-1}} \nonumber\\
       & \leq \frac{C}{t^\frac{N}{N-1}},\,\,\,\forall\, t \geq 1.
       \end{align}
       Therefore, we conclude that $\nabla u_n$ and $(-\Delta)^\frac{s}{2} u_n$ are bounded in the Marcinkiewicz space $M^\frac{N}{N-1}(\Omega)$. Hence, using the embedding in \eqref{embed2.1}, we get the sequence $(u_n)$ is bounded in the space $X_0^{s,q}(\Omega)$ for every $q < \frac{N}{N-1}.$ This completes the proof.
       \end{proof}
       \begin{theorem}\label{thm3.1<1}
           Let $0<\gamma \leq 1$. Then there exists a weak solution $u\in X_0^{s,q}(\Omega)$ to the problem \eqref{eq2:problem} for every $ q < \frac{N}{N-1}.$  
       \end{theorem}
       \begin{proof}
       From Lemma \ref{lmn3.1}, it follows that there exists $u\in X_0^{s,q}(\Omega)$ such that the sequence $(u_n)$ converges weakly to $u\in X_0^{s,q}(\Omega)$ for every $q<\frac{N}{N-1}$. That is,
       %From the above Lemma \ref{lmn3.1}, we say that there exists $u$ such that the sequence $(u_n)$ converges weakly to $u$ in $X_0^{s,q}(\Omega)$ for every $ q < \frac{N}{N-1}$. This implies that for $\phi \in C_c^\infty(\Omega)$ we have,
       \begin{equation}\nonumber
           \lim_{n\rightarrow \infty} \int_\Omega \nabla u_n \cdot \nabla \phi+\int_\Omega (-\Delta)^\frac{s}{2} u_n \cdot (-\Delta)^\frac{s}{2} \phi = \int_\Omega \nabla u \cdot \nabla \phi+\int_\Omega (-\Delta)^\frac{s}{2} u \cdot (-\Delta)^\frac{s}{2} \phi, ~\forall\,\phi \in C_c^\infty(\Omega)
       \end{equation}
       Moreover, by compactness embedding $X_0^{s,q}(\Omega)\hookrightarrow L^1(\Omega)$, we  assert that $u_n$ converges to $u$ strongly in $L^1(\Omega)$ and $u_n$ converges to $u$ {\it{a.e.}} in $\Omega$. Now take $\phi \in C_c^\infty(\Omega)$ and $K=\{x\in \Omega:\phi(x)\neq 0\} \subset supp(\phi)$. Thus from Lemma \ref{lmm2.6}, there exists $C_K$ such that $u_n\geq C_K>0$. Therefore,
       \begin{align*}
           0 \leq \bigg| \frac{f_n\phi}{(u_n+\frac{1}{n})^\gamma}\bigg| \leq \frac{\|\phi\|_{L^\infty(\Omega)}f}{C_K^\gamma}.
       \end{align*}
       Applying the Lebesgue dominated converges theorem, we get 
       \begin{equation*}
           \lim_{n\rightarrow \infty} \int_\Omega \frac{f_n\phi}{(u_n+\frac{1}{n})^\gamma} =\int_\Omega \frac{f\phi}{u^\gamma}.
       \end{equation*}
       Thus, we can pass to the limit as $n \rightarrow \infty$ in the last term of \eqref{eq2.11:SWF} involving $\mu_n$. Therefore, passing to the limit as $n\rightarrow\infty$ in \eqref{eq2.11:SWF}, we obtain that there exists a weak solution to the problem \eqref{eq2:problem} in $X_0^{1,q}(\Omega)$ for every $q < \frac{N}{N-1}$.
       \end{proof}
\subsection{Existence result when \texorpdfstring{$\gamma>1$}{γ>1}}\text{ }\\
Since $\gamma>1$ is the strong singular case, we aim to obtain local estimates on $u_n$ in $X_0^{s,2}(\Omega)$. We first establish a global estimate on $T_k^\frac{\gamma+1}{2}(u_n)$ in $X_0^{s,2}(\Omega)$ in a weak sense to make sense of the boundary values of $u$.
     \begin{lemma}\label{lmn3.2>1}
        Let $u_n$ be a weak solution of the problem \eqref{eq2.10:SWP} for $\gamma > 1$. Then  $T_k^\frac{\gamma+1}{2}(u_n)$ is bounded in $X_0^{s,2}(\Omega)$ for every fixed $k>0$. 
     \end{lemma}
     \begin{proof} Note that $\int_\Omega (-\Delta)^\frac{s}{2} u_n\cdot(-\Delta)^\frac{s}{2} T^\gamma_k(u_n)\geq 0,~\forall\, k>0$. Thus, taking $\phi=T^\gamma_k(u_n)$ as the test function in \eqref{eq2.11:SWF}, we get
         \begin{align}\label{eq3.15E}
              \int_\Omega \nabla u_n\cdot\nabla T^\gamma_k(u_n) &\leq\int_\Omega \nabla u_n\cdot\nabla T^\gamma_k(u_n)+\int_\Omega (-\Delta)^\frac{s}{2} u_n\cdot(-\Delta)^\frac{s}{2} T^\gamma_k(u_n)\nonumber\\
              &=\int_\Omega \frac{f_n T^\gamma_k(u_n)}{(u_n +\frac{1}{n})^\gamma} + \int_\Omega \mu_n T^\gamma_k(u_n) 
         \end{align}
 Also,
\begin{align}\label{eq3.16F}
    \int_\Omega \nabla u_n \cdot \nabla T^\gamma_k(u_n)&=\gamma \int_\Omega \nabla u_n\cdot\nabla T_k(u_n)T^{\gamma-1}_k(u_n)\nonumber\\
    &=\gamma \int_\Omega T^{\frac{\gamma-1}{2}}_k(u_n)\nabla u_n \cdot T^{\frac{\gamma-1}{2}}_k(u_n)\nabla T_k(u_n)\nonumber\\
    &=\gamma \int_\Omega T^{\frac{\gamma-1}{2}}_k(u_n)\nabla T_k(u_n) \cdot T^{\frac{\gamma-1}{2}}_k(u_n)\nabla T_k(u_n)\nonumber\\ 
    &=\gamma \bigg(\frac{2}{\gamma+1}\bigg)^2 \int_\Omega \nabla T^{\frac{\gamma+1}{2}}_k(u_n) \cdot \nabla T^{\frac{\gamma+1}{2}}_k(u_n)\nonumber\\
    &=\frac{4\gamma}{(\gamma+1)^2} \int_\Omega |\nabla T^{\frac{\gamma+1}{2}}_k(u_n)|^2.
\end{align}
Now we estimate the right-hand side of \eqref{eq3.15E} to obtain
\begin{align}\nonumber
    \int_\Omega \frac{f_n T^\gamma_k(u_n)}{(u_n +\frac{1}{n})^\gamma} + \int_\Omega \mu_n T^\gamma_k(u_n) & \leq \int_\Omega \frac{f_n.u_n^\gamma}{(u_n +\frac{1}{n})^\gamma} + \int_\Omega \mu_n k^\gamma\\ \nonumber
& \leq \int_\Omega f_n + k^\gamma \int_\Omega \mu_n\\ \nonumber
&\leq \int_\Omega |f|+ k^\gamma \int_\Omega \mu_n\\ \nonumber
& =\|f\|_{L^1(\Omega)}+k^\gamma \|\mu_n\|_{L^1(\Omega)}\\ 
& \leq C(k,\gamma)k^\gamma.\label{eq3.17G}
\end{align}
Therefore,
\begin{align}\label{eq3.18H}
    \int_\Omega |\nabla T^{\frac{\gamma+1}{2}}_k(u_n)|^2 \leq C(k,\gamma)k^\gamma.
\end{align}
    Similarly, using \eqref{sg-emb2}, we obtain 
    \begin{align}\label{eq3.19I}
   \int_\Omega |(-\Delta)^\frac{s}{2}T^{\frac{\gamma+1}{2}}_k(u_n)|^2\leq C(k,\gamma)k^\gamma.
    \end{align}
    Therefore,
    \begin{equation}\label{eq3.20J}
    \int_\Omega |\nabla T^{\frac{\gamma+1}{2}}_k(u_n)|^2 +\int_\Omega |(-\Delta)^\frac{s}{2}T^{\frac{\gamma+1}{2}}_k(u_n)|^2\leq C(k,\gamma)k^\gamma.
    \end{equation}
    Hence, we get $T^{\frac{\gamma+1}{2}}_k(u_n)$ is bounded in $X_0^{s,2}(\Omega)$ for all $k>0$, completing the proof.
    \end{proof}
    The next lemma consisting of local estimates on $u_n$ plays a crucial role in obtaining the existence result as it allows us to pass the limit in \eqref{eq2.11:SWF} as  $n \rightarrow \infty$.
    \begin{lemma}\label{lmn3.3>1}
        Let $u_n$ be a solution to the problem \eqref{eq2.10:SWP} for $\gamma > 1$. Then $u_n$ is bounded in $X_{loc}^{s,q}(\Omega)$ for every $q< \frac{N}{N-1}$.
    \end{lemma}
    \begin{proof}
    Recall that $u_n=T_1(u_n)+G_1(u_n).$ Thus we divide the proof into two parts showing $G_1(u_n)$ and $T_1(u_n)$ are bounded in $X_0^{s,q}(\Omega)$ for every $q< \frac{N}{N-1}$.\\
        \textbf{Step I:} We claim $G_1(u_n)$ is bounded in $X_0^{s,q}(\Omega)$ for every $q< \frac{N}{N-1}$.\\
        By the Definition $G_k$ and $T_k$, we have
        
\iffalse
        \begin{align}\label{eq3.21}
            G_1(u_n)=(u_n-1)^+=\begin{cases}
             0, \,\,\,\,\,\,\,\,&~\text{if}\,\,\, 0 \leq u_n \leq 1, \\ 
             (u_n -1),&~\text{if}\,\,\, u_n >1.
            \end{cases}
        \end{align} 
\fi
        
        \begin{align}\label{eq3.22}
            \nabla G_1(u_n)&= \begin{cases}
               0 \,\,\,\,\,\,\,\,&~\text{if}\,\,\, 0 \leq u_n \leq 1 \\ 
               \nabla(u_n -1)&~\text{if}\,\,\, u_n >1
            \end{cases} \nonumber\\
            &= \begin{cases}
               0\,\,\,\,\,\,\,\,&~\text{if}\,\,\, 0 \leq u_n \leq 1, \\ 
               \nabla u_n&~\text{if}\,\,\, u_n >1,
            \end{cases}
        \end{align}
        and 
        \begin{align}\label{eq3.26C}
            \nabla T_k(G_1(u_n))=\begin{cases}
                \nabla u_n &~ \text{if} ~ 1<u_n \leq k+1, \\
                0 &~ \text{if} ~ u_n \leq 1 ~\text{or}~ u_n>k+1.
            \end{cases}
        \end{align}
        Similarly,
          \begin{align}\label{eq3.23}
            (-\Delta)^\frac{s}{2}G_1(u_n)&= \begin{cases}
               0, \,\,\,\,\,\,\,\,&~\text{if}\,\,\, 0 \leq u_n \leq 1, \\ 
               (-\Delta)^\frac{s}{2} u_n&~\text{if}\,\,\, u_n >1.
            \end{cases}
        \end{align}
        and
        \begin{align}\label{eq3.27D}
            (-\Delta)^\frac{s}{2} T_k(G_1(u_n))=\begin{cases}
                (-\Delta)^\frac{s}{2} u_n& ~ \text{if} ~ 1<u_n \leq k+1, \\
                0 &~ \text{if} ~ u_n, \leq 1 ~\text{or}~ u_n>k+1.
            \end{cases}
        \end{align}
        Choose $\phi=T_k(G_1(u_n))$, $k>1$ as a test function in \eqref{eq2.11:SWF}. Then we get
        \begin{align*}
    &\int_\Omega |\nabla T_k(G_1(u_n))|^2+\int_\Omega |(-\Delta)^\frac{s}{2} T_k(G_1(u_n))|^2 \nonumber \\ 
    &=\int_{ 1<u_n \leq k+1} |\nabla (u_n)|^2+\int_{1<u_n \leq k+1} |(-\Delta)^\frac{s}{2} (u_n)|^2 \nonumber \\ 
    & =\int_\Omega \nabla u_n \cdot \nabla T_k(G_1(u_n))+\int_\Omega (-\Delta)^\frac{s}{2} u_n \cdot (-\Delta)^\frac{s}{2} T_k(G_1(u_n)) \nonumber \\ 
    & = \int_\Omega \frac{f_n T_k(G_1(u_n))}{(u_n +\frac{1}{n})^\gamma} + \int_\Omega \mu_n T_k(G_1(u_n))\nonumber\\
    &= \int_{\left(u_n +\frac{1}{n}\right)\leq k}\frac{f_n T_k(G_1(u_n))}{(u_n +\frac{1}{n})^\gamma} +\int_{(u_n +\frac{1}{n}) \geq k} \frac{f_n T_k(G_1(u_n))}{(u_n +\frac{1}{n})^\gamma} +\int_\Omega \mu_n T_k(G_1(u_n)) \nonumber \\ 
    & \leq \int_{\left(u_n +\frac{1}{n}\right)\leq k}\frac{f_n T_k(G_1(u_n))}{(u_n +\frac{1}{n})^\gamma}+\int_{(u_n +\frac{1}{n})\geq k}\frac{f k}{k^\gamma} +k\|\mu_n\|_{L^1(\Omega)} \nonumber \\ 
   &=\int_{1<u_n<\left(u_n +\frac{1}{n}\right)\leq k}\frac{f_n T_k(G_1(u_n))}{(u_n + \frac{1}{n})^\gamma}+k^{1-\gamma} \int_{(u_n +\frac{1}{n})\geq k} f +k\|\mu_n\|_{L^1(\Omega)} \nonumber \\ 
   & \leq \int_{\left(u_n +\frac{1}{n}\right)\leq k} f k+k^{1-\gamma} \int_{(u_n +\frac{1}{n})\geq k} f +k\|\mu_n\|_{L^1(\Omega)} \nonumber \\ 
   & \leq Ck, \,\,\,\forall \,k > 1.
\end{align*}
   Thus, we obtain 
   \begin{align}\label{eq3.31H}
        &\int_\Omega |\nabla T_k(G_1(u_n))|^2+\int_\Omega |(-\Delta)^\frac{s}{2} T_k(G_1(u_n))|^2 
       \leq Ck, \,\,\,\forall \,k \geq 1.
   \end{align}
    Moreover, from \eqref{eq3.20J} and the mixed Sobolev inequality \eqref{eq2.8S} with $p=2$, we obtain
   \begin{align}\label{eq3.33}
       \frac{1}{C^2}\bigg(\int_\Omega |T_k^\frac{\gamma+1}{2}(u_n)|^{2^*}\bigg)^\frac{2}{{2^*}} & \leq\int_\Omega |\nabla T^\frac{\gamma+1}{2}_k(u_n)|^2+\int_\Omega |(-\Delta)^\frac{s}{2}T^\frac{\gamma+1}{2}_k(u_n)|^2\nonumber\\
       &\leq C(k,\gamma)k^\gamma, ~\forall\, k \geq 1.
      \end{align}
         We now proceed as in Lemma \ref{lmn3.1} to show that $\nabla G_1(u_n)$ and $(-\Delta)^\frac{s}{2}G_1(u_n)$ are bounded in the Mareinkiewicz space $M^\frac{N}{N-1}(\Omega)$. Clearly, the following inclusions hold.
         \begin{align}\nonumber
           {\{|\nabla u_n| \geq t, u_n>1}\} & ={\{|\nabla u_n| \geq t,1<u_n \leq k+1}\} \cup {\{|\nabla u_n| \geq t,u_n > k+1}\}\\ \nonumber
& \subset {\{|\nabla u_n| \geq t,1<u_n \leq k+1}\} \cup {\{u_n > k+1}\} \subset \Omega.
        \end{align}
        and
        \begin{align}\nonumber
            \left\{|(-\Delta)^\frac{s}{2} u_n| \geq t, u_n>1\right\}
& \subset {\{|(-\Delta)^\frac{s}{2} u_n| \geq t,1<u_n \leq k+1}\} \cup {\{u_n > k+1}\} \subset \Omega.
        \end{align}
        Thus, by using the properties of Lebesgue measure $m$, we get
        \begin{align}\label{eq3.24A}
            m({\{|\nabla u_n| \geq t, u_n>1}\})\leq m({\{|\nabla u_n| \geq t,1<u_n \leq k+1}\}) + m({\{u_n > k+1}\})
        \end{align}
        and
        \begin{align}\label{eq3.25B}
            m({\{|(-\Delta)^\frac{s}{2} u_n| \geq t, u_n>1}\})\leq & m({\{|(-\Delta)^\frac{s}{2} u_n| \geq t,1<u_n \leq k+1}\})+ m({\{u_n > k+1}\}). 
        \end{align}
 Let us denote
  \begin{align*}
      J_1&={\{|\nabla u_n| \geq t,1<u_n \leq k+1}\},\\
      J_2&={\{|(-\Delta)^\frac{s}{2} u_n| \geq t,1<u_n \leq k+1}\}~\text{and}\\
      J_3&=\{x\in \Omega: u_n>k+1\}.
  \end{align*} 
  Now on using \eqref{eq3.31H} and the definition of $J_1$, we get
   \begin{align}\label{eq3.34I}
        m({\{|\nabla u_n| \geq t,1<u_n \leq k+1}\}) & \leq \frac{1}{t^2} \int_{J_1} |\nabla u_n|^2 \nonumber \\ 
        &  \leq \frac{1}{t^2}\left[ \int_{1<u_n \leq k+1}|\nabla u_n|^2 + \int_{1<u_n \leq k+1}|(-\Delta)^\frac{s}{2}u_n|^2\right] \nonumber \\ 
      &\leq \frac{Ck}{t^2}, \,\,\forall\, k ~\geq 1. 
   \end{align}
  Similarly, on $J_2$, we have
   \begin{align}\label{eq3.35J}
        m({\{|(-\Delta)^\frac{s}{2} u_n| \geq t,1<u_n \leq k+1}\}) & \leq \frac{1}{t^2} \int_{J_2} |(-\Delta)^\frac{s}{2} u_n|^2 \nonumber \\ 
        &  \leq \frac{1}{t^2}\left[ \int_{1<u_n \leq k+1}|\nabla u_n|^2 + \int_{1<u_n \leq k+1}|(-\Delta)^\frac{s}{2}u_n|\right] \nonumber \\ 
      &=\frac{Ck}{t^2}, \,\,\forall\, k \geq 1.
   \end{align}
  Finally, using the fact $T_k(u_n)=k$ on $J_3$ and using \eqref{eq3.33}, we obtain
       \begin{align}\nonumber
            C^2Ck^\gamma \geq \bigg(\int_{J_3} |T_k^\frac{\gamma+1}{2}(u_n)|^{2^*}\bigg)^\frac{2}{{2^*}}=\left( \int_{J_3} k^\frac{(\gamma+1){2^*}}{2} \right)^\frac{2}{2^*} = k^{\gamma+1}m(\{u_n>k+1\})^\frac{2}{2^*},~\forall~ k \geq 1,
            \end{align}
            which implies
            \begin{align}\label{eq3.36k}
                m(\{u_n>k+1\})\leq \frac{C}{k^\frac{2^*}{2}} =\frac{C}{k^\frac{N}{N-2}},\,\,\forall \, k \geq 1.
            \end{align}
        Thus we get $(u_n-1)$ is  bounded in the Marcinkiewicz Space $M^\frac{N}{N-2}(\Omega)$ which implies $G_1(u_n)$ is bounded in the Marcinkiewicz Space $M^\frac{N}{N-2}(\Omega)$. \\
 Now use the estimates \eqref{eq3.34I}, \eqref{eq3.36k} in \eqref{eq3.24A} and then put $k=t^\frac{N-2}{N-1}$ to obtain
       \begin{align}\label{eq3.37L}
            m({\{|\nabla u_n| > t, u_n>1}\}) &\leq m({\{|\nabla u_n| \geq t,1<u_n \leq k+1}\}) + m({\{u_n > k+1}\}) \nonumber \\
             & \leq \frac{Ck}{t^2} + \frac{C}{k^\frac{N}{N-2}}\nonumber\\
             &\leq \frac{C}{t^\frac{N}{N-1}}\,\,\forall \,t \geq 1.
       \end{align}
       Again, using the estimates \eqref{eq3.35J}, \eqref{eq3.36k} in \eqref{eq3.25B} and then putting $k=t^\frac{N-2}{N-1}$, we deduce
       \begin{align}\label{eq3.38M}
           m({\{|(-\Delta)^\frac{s}{2} u_n| > t, u_n>1}\}) & \leq m({\{|(-\Delta)^\frac{s}{2} u_n| \geq t,1<u_n \leq k+1}\}) + m({\{u_n > k+1}\}) \nonumber \\
             & \leq \frac{Ck}{t^2} + \frac{C}{k^\frac{N}{N-2}}\nonumber\\
             & \leq \frac{C}{t^\frac{N}{N-1}}\,\,\forall\, t \geq 1.
       \end{align}
       Therefore, we get $\nabla G_1(u_n)$ and $(-\Delta)^\frac{s}{2}G_1(u_n)$ are bounded in $M^\frac{N}{N-1}(\Omega)$. Thanks to the embedding \ref{embed2.1}, we conclude that $G_1(u_n)$ is  bounded in $X_0^{s,q}(\Omega)$ for every $q<\frac{N}{N-1}$.\\
       \textbf{Step II:} $T_1(u_n)$ is bounded in $X_{loc}^{s,q}(\Omega)$ for every $q<\frac{N}{N-1}.$ \\
       In this case, we focus on $u_n\leq 1$ and under this assumption, we prove that
      \begin{equation}\label{T_1 bdd}
          \int_K |\nabla T_1(u_n)|^2+\int_K |(-\Delta)^\frac{s}{2} T_1(u_n)|^2 \leq C, ~\forall\, K \Subset \Omega,
        \end{equation}
       where $C>0$ is constant. Putting, $\phi=T^\gamma_1(u_n)$ as a test function in \eqref{eq2.11:SWF}, we get
       \begin{align}
           \int_\Omega \nabla u_n \cdot \nabla T^\gamma_1(u_n) & \leq \int_\Omega \nabla u_n \cdot \nabla T^\gamma_1(u_n)+\int_\Omega (-\Delta)^\frac{s}{2} u_n \cdot (-\Delta)^\frac{s}{2} T^\gamma_1(u_n)\nonumber \\
          & =\int_\Omega \frac{f_n T^\gamma_1(u_n)}{(u_n +\frac{1}{n})^\gamma} + \int_\Omega \mu_n T^\gamma_1(u_n)\nonumber\\
          & \leq \int_\Omega \frac{f_n u_n^\gamma}{(u_n +\frac{1}{n})^\gamma} + \int_\Omega \mu_n \nonumber\\ 
& \leq \int_\Omega f_n + \int_\Omega \mu_n \nonumber\\ 
&\leq \int_\Omega |f|+ \int_\Omega \mu_n \nonumber\\ 
& =\|f\|_{L^1(\Omega)}+\|\mu_n\|_{L^1(\Omega)}\nonumber\\
&\leq C. \label{eq3.44C}
             % \Rightarrow  \int_\Omega \nabla u_n \cdot \nabla T^\gamma_1(u_n) \leq \int_\Omega \frac{f_n T^\gamma_1(u_n)}{(u_n +\frac{1}{n})^\gamma} + \int_\Omega \mu_n T^\gamma_1(u_n).
      \end{align}
    Now $\forall\, K \Subset \Omega$, we have $u_n \geq C_K>0$. Therefore, the left-hand side of \eqref{eq3.44C} can be estimated as
     \begin{align}\label{eq3.43B}
         \int_\Omega \nabla u_n \cdot \nabla T^\gamma_1(u_n) & =\gamma\int_\Omega \nabla u_n \cdot \nabla T_1(u_n) T_1^{\gamma-1}(u_n)\nonumber\\
        & =\gamma\int_\Omega |\nabla T_1(u_n)|^2 T_1^{\gamma-1}(u_n) \nonumber\\ 
         & \geq \gamma\int_K |\nabla T_1(u_n)|^2 T_1^{\gamma-1}(u_n) \nonumber\\
         & \geq \gamma  C_K^{\gamma- 1} \int_K |\nabla T_1(u_n)|^2.
      \end{align}
      Thus, using the inequalities \eqref{eq3.44C} and \eqref{eq3.43B}, we get
      \begin{align}
          \int_K |\nabla T_1(u_n)|^2 \leq C.
      \end{align}
     Therefore, using \eqref{sg-emb}, we conclude that 
      \begin{align}
       \int_K |\nabla T_1(u_n)|^2+\int_K |(-\Delta)^\frac{s}{2} T_1(u_n)|^2 \leq C,   
      \end{align}
      which implies $u_n\in X_{loc}^{s,q}(\Omega)$ is bounded for every $q< \frac{N}{N-1}$. This completes the proof. 
    \end{proof}
\begin{theorem}\label{thm3.2>1}
     Let $\gamma > 1$. Then there exists a weak solution $u\in X_{loc}^{s,q}(\Omega)$ to the problem \eqref{eq2:problem} for every $ q < \frac{N}{N-1}.$ 
\end{theorem}
\begin{proof}
 Proceeding with similar arguments as in the Theorem \ref{thm3.1<1} combining with the results obtained in Lemma \ref{lmn3.2>1} and Lemma \ref{lmn3.3>1}, we conclude the existence of a weak solution to the problem \eqref{eq2:problem} in $ X_{loc}^{s,q}(\Omega)$ for every $ q < \frac{N}{N-1}.$   
\end{proof}
\noindent \textit{Proof of Theorem \ref{t1.1}}: The result follows from Theorem \ref{thm3.1<1} and Theorem \ref{thm3.2>1}.\\

We now turn our attention to the existence of a \emph{veryweak} solution to the problem \eqref{eq2:problem}. Prior to the existence of a \emph{veryweak} solution to \eqref{eq2:problem}, we consider the following measure data problem for the existence of a $\it{duality~solution}$.
\section{Duality solution to the mixed local-nonlocal elliptic PDEs with measure data}\label{sec4}
In this section, we prove the existence of a ``$\it{duality~solution}$" to the following boundary value problem,
 \begin{equation}\label{eq4.1 DP}
\begin{aligned}
   -\Delta u+(-\Delta)^s & u =\mu \,\, \text{in} \,\,\Omega,\\
  & u=0\,\, \text{in}\,\,\mathbb  R^N \setminus \Omega,
   \end{aligned}
\end{equation}
where $\Omega$ is an open bounded subset of $\mathbb R^N$ with Lipschitz boundary $\partial \Omega$, $0<s<1$, $N>2s$ and $\mu$ is any bounded Radon measure.

\par We first prove the existence and uniqueness of a $\it{duality~solution}$ to the problem \eqref{eq4.1 DP}.
\begin{definition}\label{dual def}
    We say $u \in L^1(\mathbb R^N)$ is a $duality~ solution$ to the problem \eqref{eq4.1 DP} if $u=0$ in $\mathbb R^N \setminus \Omega$ and
    \begin{equation}\label{eq4.2 DF}
        \int_\Omega ug \d x= \int_\Omega w \d \mu,~ \forall~ g \in C_c^\infty(\Omega),
    \end{equation}
     where $w$ is the weak solution to the following problem
    \begin{equation}\label{eq4.3 DWF}
    \begin{aligned}
   -\Delta w+(-\Delta)^s & w =g \,\, \text{in} \,\,\Omega,\\
  & w=0\,\, \text{in}\,\,\mathbb  R^N \setminus \Omega.
   \end{aligned}
\end{equation}
\end{definition}
\noindent  Note that if $w\in C(\bar{\Omega})$ (or, atleast $w\in L^\infty(\Omega)$ in case of $ \mu \in L^1(\Omega)$), the equation \eqref{eq4.2 DF} is well-defined. For the existence of a continuous solution of \eqref{eq4.3 DWF}, we refer to Remark \ref{RM4.2}. In particular, from \cite[Theorem 1.3, Theorem 1.4]{SVWZ2023}, we have$w\in C_{loc}^{2,\alpha}(\Omega)\cap C^{1,\alpha}(\bar{\Omega})$.\\

The next lemma provides a relation between the weak and $\it{duality}$ solutions.
\begin{lemma}\label{lmm 4.1}
    Every {\it{duality}} solution to the problem \eqref{eq4.1 DP} is a weak solution to the problem \eqref{eq4.1 DP} and vice-versa. 
\end{lemma}

\begin{proof}
   Let $u$ be a weak solution of the problem \eqref{eq4.1 DP}. Then $u=0$ in $\mathbb R^N \setminus \Omega$ and
     $$\int_\Omega \nabla u\cdot\nabla \phi \d x+\int_\Omega (-\Delta)^\frac{s}{2} u\cdot(-\Delta)^\frac{s}{2} \phi \d x= \int_\Omega \phi \d \mu,\,\,\,\,\,\, \forall ~\phi \in C_c^\infty(\Omega).$$
     Now for any $g\in C_c^\infty(\Omega)$, let $w$ be the corresponding solution of \eqref{eq4.3 DWF}. We have $u=0$ in $\mathbb R^N\setminus \Omega$ and on using the integration by parts, we get
     \begin{align}
         \int_\Omega ug\d x & =\int_\Omega \nabla u\cdot\nabla w \d x+\int_\Omega (-\Delta)^\frac{s}{2} u \cdot (-\Delta)^\frac{s}{2} w \d x, \nonumber \\ 
         & = \int_\Omega w d\mu. \nonumber
     \end{align}
     Therefore, $u$ is a {\it{duality}} solution of the problem \eqref{eq4.1 DP}.\\
     Conversely, let $u$ be a {\it{duality}} solution of the problem \eqref{eq4.1 DP} in the sense of Definition \ref{dual def} and let $\bar u$ be a weak solution of the problem $\eqref{eq4.1 DP}$. Then using the above argument, we conclude that $\bar u$ is also a {\it{duality}} solution of the problem \eqref{eq4.1 DP}. Thus, we get 
     \begin{equation}\label{eq4.4}
         \int_\Omega ug\d x=\int_\Omega w d\mu =\int_\Omega \bar{u}g\d x, ~ \,\, \forall~ g \in C_c^\infty(\Omega).
     \end{equation}
     On using \eqref{eq4.4}, we obtain
     $$ \int_\Omega (u-\bar{u})g\d x=0,~\,\, \forall~ g \in C_c^\infty(\Omega).$$
     Therefore, implies that $u=\bar{u}$ {\it{a.e.}} in $\Omega$. This completes the proof.
\end{proof}
\noindent Let us recall the following two important theorems due to Biagi {\it{et al.}} \cite{BDVV2022} and Su {\it{et al.}} \cite{SVWZ2023}.
\begin{theorem}{\cite[Theorem 4.7]{BDVV2022}}\label{thm4.1}
    Let $\Omega$ is a bounded open subset of $\mathbb R^N$ with $C^1$ boundary $\partial \Omega$, $f\in L^p(\Omega)$ with $p>\frac{N}{2}$ and $u$ be the weak solution of the following problem.
    \begin{align}
         -\Delta u+(-\Delta)^s & u =f \,\, \text{in} \,\,\Omega, \nonumber \\ 
  & u=0\,\, \text{in}\,\,\mathbb  R^N \setminus \Omega. \nonumber
    \end{align}
    Then $u\in L^\infty(\mathbb R^N)$ and there exists constant $C>0$ such that 
    $$\|u\|_{L^\infty(\mathbb R^N)} \leq C\|f\|_{L^p(\Omega)}.$$
\end{theorem}
\begin{theorem}{\cite[Theorem 4.1]{SVWZ2023}}\label{thm4.2}
    Let $\Omega$ be a $C^{1,1}$ domain in $\mathbb R^N$, $s\in(0,1)$ and $f \in L^p(\Omega)$, where $p\in(1,\infty)$ if $s\in(0,\frac{1}{2}]$ and $p\in(N,\frac{N}{2s-1})$ if $s\in(\frac{1}{2},1).$ Then the following problem 
\begin{equation}\nonumber
-\Delta u+(-\Delta)^s u =f \,\, \text{in} \,\,\Omega,
\end{equation}
has a unique solution $u \in W^{2,p}(\Omega) \cap W_0^{1,p}(\Omega)$. Moreover, there exists $C(\Omega,N,s,p)$ such that $\|u\|_{W^{2,p}(\Omega)}\leq C (\|u\|_{L^p(\Omega)}+\|f\|_{L^p(\Omega)}).$
\end{theorem}
\begin{remark}\label{RM4.2}
     Let $2>\frac{N}{p}$. By the Sobolev embedding theorem \cite{EVANS2022}, we have,
     $$W^{2,p}(\Omega) \subset C^k(\bar{\Omega}),~\,\,\, k=~\text{integral part of}~\bigg(2-\frac{N}{p}\bigg).$$
    Therefore, $u \in C(\bar{\Omega})$, where $p\in(\frac{N}{2},\infty)$ if $s\in(0,\frac{1}{2}]$ and $p\in(N,\frac{N}{2s-1})$ if $s\in(\frac{1}{2},1).$
\end{remark}
With the above results, we now state and prove the existence and uniqueness of $\it{duality~ solution}$ to the problem \eqref{eq4.1 DP}. 
\begin{theorem}\label{thm4.3 DS}
    There exists a unique $\it{duality~solution}$ $u\in L^1(\Omega)$ for the problem \eqref{eq4.1 DP} in the sense of Definition \ref{dual def}. Moreover, $u\in L^q(\Omega)$, where $q\in(1,\frac{N}{N-2})$ if $s\in(0,\frac{1}{2}]$ and $q\in(\frac{N}{N-2s+1},\frac{N}{N-1})$ if $s\in(\frac{1}{2},1).$
\end{theorem}
\begin{proof}
    We prove the existence by dividing it into two cases. For any $p\in [1,\infty)$, using the density of $C^\infty_c(\Omega)$ in $L^p(\Omega)$, we get that the Definition \ref{dual def} holds for any $g\in L^p(\Omega)$.\\
    \par \textbf{Case I:} If $0<s \leq \frac{1}{2}$. Fix $p \in (\frac{N}{2},\infty)$. For any $g\in L^p(\Omega)$, we define $T:L^p(\Omega)\rightarrow \mathbb R$ as 
    $$T(g)=\int_\Omega w(x) \d \mu.$$
    Then by the using of Theorem \ref{thm4.1}, and Remark \ref{RM4.2}, we deduce that $T$ is well-defined and we get 
    $$|T(g)|\leq \|w\|_{L^\infty(\Omega)}|\mu|(\Omega)\leq C(\Omega,\mu,N,s,p) \|g\|_{L^p(\Omega)}.$$
    Thus, T is a bounded linear functional (continuous) on $L^p(\Omega)$. By Riesz representation theorem, there exists a unique function $u \in L^q(\Omega),$ $q=\frac{p}{p-1}$ such that
    \begin{equation}\label{eq4.5L}
        \int_\Omega ug \d x= \int_\Omega w \d \mu.
    \end{equation}
    We can repeat the above argument for any $p \in (\frac{N}{2},\infty)$ to obtain a unique $u \in L^q(\Omega),$ with $1<q<\frac{N}{N-2}$ such that \eqref{eq4.5L} holds.\\
    
    \par \textbf{Case II:} If $\frac{1}{2}<s<1$. Fix $p\in (N,\frac{N}{2s-1})$. For any $g\in L^p(\Omega)$, we define $T:L^p(\Omega)\rightarrow \mathbb R$ as
    $$T(g)=\int_\Omega w(x) \d \mu.$$
   Proceeding with similar arguments as in \textbf{Case I}, we obtain that T is a bounded linear functional (continuous) on $L^p(\Omega)$. Thus, the Riesz Representation theorem guarantees there exists a unique function $u \in L^q(\Omega)$ with $q=\frac{p}{p-1}$ satisfying the equation \eqref{eq4.5L}. Again repeating this argument for $p\in (N,\frac{N}{2s-1}),$ we get a unique $u \in L^q(\Omega)$ with $\frac{N}{N-2s+1}<q<\frac{N}{N-1}$ such that \eqref{eq4.5L} holds.\\
    It remains to prove the uniqueness result. Let $u$ and $v$ be two {\it{duality}} solutions of the problem \eqref{eq4.1 DP} in the sense of Definition \ref{dual def}. Therefore, we have
    $$\int_\Omega (u-v)g \d x=0, ~ \,\,\, \forall~ g\in C_c^\infty(\Omega).$$
    This implies that $u=v$ a.e $\Omega$. Therefore, the {\it{duality}} solution exists uniquely. This completes the proof.
\end{proof}
Let us consider $\mu$ a nonnegative, bounded Radon measure in Theorem \ref{thm4.3 DS}. Then we have the following Corollary related to the following problem,
 \begin{equation}\label{eq5.5NDP}
\begin{aligned}
   -\Delta u+(-\Delta)^s & u =\mu \,\, \text{in} \,\,\Omega,\\
  & u=0\,\, \text{in}\,\,\mathbb  R^N \setminus \Omega.
   \end{aligned}
\end{equation}
\begin{cor}\label{RM5.1}
     Theorem \ref{thm4.3 DS} holds for the problem \eqref{eq5.5NDP}, that is the problem \eqref{eq5.5NDP} possesses a \emph{duality} solution. Moreover, every $\it{duality}$ solution to the problem \eqref{eq5.5NDP} is a $\it{veryweak}$ solution in the sense of Definition \ref{def5.1 VW}. 
     \end{cor}

\noindent \textit{Proof of Theorem \ref{t1.2}}: The result is straight forward from Lemma \ref{lmm 4.1} and  Theorem \ref{thm4.3 DS}.\\

\noindent We are now ready with sufficient tools to dive into the proof of the existence of a $\it{veryweak}$ solution to the problem \eqref{eq2:problem}. Note that  Theorem \ref{thm4.3 DS} remain true if we replace $\mu$ by $f+\mu$, where $0\leq f\in L^1(\Omega)$.

\section{Existence of veryweak solution to the problem \ref{eq2:problem} for \texorpdfstring{$0<\gamma<1$.}{0<γ<1}}\label{sec5}
In this section, we  assume $\Omega$ is a bounded domain of class $C^{2,\beta}$ for some $0<\beta<1$. We consider the following elliptic problem,
\begin{align}\label{eq5.1p<1}
   -\Delta u+(-\Delta)^s u &= \frac{f(x)}{u^\gamma} +\mu \,\, \text{in} \,\,\Omega,\\ 
   u&=0\,\, \text{in}\,\,\mathbb  R^N \setminus \Omega, \nonumber
  \end{align}
where $0 < \gamma< 1 $, $0<s<1$, $f(x)>0,\,\forall~ x\in \Bar{\Omega}$ such that $f\in C^{\beta}(\Bar{\Omega})$ and $\mu$ is a nonnegative, bounded Radon measure over $\Omega$. Observe that, we can give a suitable notion of solution to the problem \eqref{eq5.1p<1} by considering only the integrable functions. We fix $0<\gamma<1$ and define the ``{\it{veryweak}} solution''  as follows.
\begin{definition}\label{def5.1 VW}
A {\it{veryweak}} solution to the problem \eqref{eq5.1p<1} is a function $u \in L^1(\Omega)$ such that $u>0$ {\it{a.e.}} in $\Omega$, $fu^{-\gamma}\in L^1(\Omega)$ and 
\begin{align}\label{eq5.2 Vw}
    \int_\Omega u (-\Delta)\phi+\int_\Omega u (-\Delta)^s\phi=\int_\Omega \frac{f\phi}{u^\gamma} +\int_\Omega \phi \d \mu, ~ \ \forall\, \phi \in C^2_0(\Bar{\Omega}).
   % \text{such that}\,\,\, (-\Delta)^s \phi \in L^\infty(\Omega),\nonumber
\end{align}
\end{definition}
We prove the existence of a {\it{veryweak}} solution to the problem \eqref{eq5.1p<1} by the ``sub-super solution'' method. We define the {\it{veryweak subsolution}} and {\it{veryweak supersolution}} to \eqref{eq5.1p<1} as below.
\begin{definition}\label{def5.2SSS}
    A {\it{veryweak subsolution}} to the problem \eqref{eq5.1p<1} is a function $\underline u\in L^1(\Omega)$ such that $ \underline u>0$ in $\Omega$, $f\underline u^{-\gamma}\in L^1(\Omega)$ and for all nonnegative $\phi \in C^2_0(\Bar{\Omega})$,
    \begin{align}\label{eq5.3sub}
    \int_\Omega \underline u (-\Delta)\phi+\int_\Omega \underline u (-\Delta)^s\phi \leq \int_\Omega \frac{f\phi}{\underline u^\gamma} +\int_\Omega \phi \d \mu.\\ \nonumber
     %\text{such that}\,\,\, (-\Delta)^s \phi \in L^\infty(\Omega),
\end{align}
Similarly, a {\it{veryweak supersolution}} to the problem \eqref{eq5.1p<1} is a function $\Bar{u}\in L^1(\Omega)$ such that $ \Bar{ u} >0$ in $\Omega$, $f\bar{u}^{-\gamma}\in L^1(\Omega)$ and for all nonnegative $\phi \in C^2_0(\Bar{\Omega})$,
 \begin{align}\label{eq5.4sup}
    \int_\Omega \bar u (-\Delta)\phi+\int_\Omega \bar u (-\Delta)^s\phi \geq \int_\Omega \frac{f\phi}{\bar u^\gamma} +\int_\Omega \phi \d \mu.
    %\text{such that}\,\,\, (-\Delta)^s \phi \in L^\infty(\Omega),
    \end{align}
\end{definition}
Now we are focused on the existence of the \emph{veryweak} solution to the problem \eqref{eq5.1p<1}.
\begin{theorem}\label{thm5.1}
    Let $\underline u$ be a {\it{veryweak subsolution}} and $\bar u$ be a {\it{veryweak supersolution}} of the equation \eqref{eq5.1p<1} such that $\underline u \leq \bar u$ in $\Omega$. Then there exists a {\it{veryweak}} solution $u$ to the problem \eqref{eq5.1p<1}, in the sense of Definition \ref{def5.1 VW} such that $\underline u \leq u \leq \bar u$.
\end{theorem}
\begin{proof}
    With the notion of $\underline u>0$ and $\bar u$ as in the Definition \ref{def5.2SSS}, we truncate the singular term of the problem \eqref{eq5.1p<1} with a nonlinearity $\bar g(\cdot,\cdot):\Omega \cross \mathbb R \rightarrow \mathbb R$ which is given by
    $$\bar g(x,t)=\begin{cases}
        & f(x)\underline u(x)^{-\gamma}, \,\,\, \text{if}\,\,\, t < \underline u(x),\\
        & f(x)t^{-\gamma}, \,\,\,  \text{if}\,\,\, \underline u(x) \leq t \leq \bar u(x),\\
       & f(x)\bar u(x)^{-\gamma}, \,\,\,  \text{if}\,\,\, t> \bar u(x).
    \end{cases}$$
     Note that $\bar g$ is well-defined {\it{a.e.}} in $\Omega$. Moreover, using  $f\underline u^{-\gamma}$ and $f\bar u^{-\gamma}$ $\in$ $L^1(\Omega)$  we have $\bar g(x,v(x)) \in L^1(\Omega)$ for every $v \in L^1(\Omega)$. We now proceed to complete the proof. Consider the problem
     \begin{align}\label{eq5.6G}
-\Delta u+(-\Delta)^s  u &= \bar g(x,u)+\mu \,\, \text{in} \,\,\Omega,  \\ 
   u&=0\,\, \text{in}\,\,\mathbb  R^N \setminus \nonumber \Omega.
 \end{align}
    \noindent \textbf{Claim 1:} Let $u$ solves \eqref{eq5.6G} in the {\it{veryweak}} sense, then we have $\underline u \leq u \leq \bar u$. This serves the purpose of concluding that $u$ is a {\it{veryweak}} solution to the problem \eqref{eq5.1p<1} in the sense of Definition \ref{def5.1 VW} as $\bar g(.,u)=g(u)=fu^{-\gamma} \in L^1(\Omega)$. Now, the {\it{veryweak}} formulation for the truncated problem \eqref{eq5.6G} is given by
    \begin{align}\label{eq5.7}
    \int_\Omega u (-\Delta)\phi+\int_\Omega u (-\Delta)^s\phi=\int_\Omega \bar g(x,u) \phi +\int_\Omega \phi \d \mu, ~  \forall\, \phi \in C^2_0(\Bar{\Omega}). 
   % \text{such that}\,\,\, (-\Delta)^s \phi \in L^\infty(\Omega), 
   \end{align} 
    To establish the claim, we first prove $u\leq \bar u$ in $\Omega$, where $u$ is a solution to \eqref{eq5.6G} and $\bar u$ is a {\it{veryweak supersolution}} to \eqref{eq5.1p<1}. Now for all nonnegative $\phi \in C_0^2(\bar\Omega)$, we subtract \eqref{eq5.4sup} from \eqref{eq5.7} to get
    %$\forall$ $\phi \in C_0^2(\Omega)$ such that $\phi\geq 0$, 
     %and $(-\Delta )^s \phi \in L^\infty(\Omega)$ 
    \begin{align}
    \int_\Omega (u-\bar u) (-\Delta)\phi+\int_\Omega (u- \bar u) (-\Delta)^s\phi & \leq \int_\Omega \bigg( \bar g(x,u)-\frac{f}{\bar u^\gamma}\bigg)\phi \nonumber\\ 
    & = \int_\Omega \chi_{\{u\leq \bar u\}}\bigg( \bar g(x,u)-\frac{f}{\bar u^\gamma}\bigg)\phi.
    \end{align}
    On applying the Kato-type inequality \eqref{keto inq 2}, we obtain
    $$0\leq\int_\Omega (u-\bar u)^+ \leq  \int_\Omega \chi_{\{u\leq \bar u\}}\bigg( {\bar g(x,u)}-\frac{f}{\bar u^\gamma}\bigg)(sign_+(u-\bar u))\phi_0=0, $$
    where $\phi_0\in C_0^2(\bar \Omega)$ satisfy $((-\Delta)+(-\Delta)^s)\phi_0=1$ in $\Omega$ and $\phi=0$ in $\mathbb R^N \setminus \Omega.$ Thus, we have
    $$ \int_\Omega (u-\bar u)^+ = 0,$$
    which implies that $u \leq \bar u$ {\it{a.e.}} in $\Omega$. One can proceed on similar lines to obtain $\underline u \leq u$ in $\Omega$, where $\underline u$ is a {\it{veryweak subsolution}} to \eqref{eq5.1p<1}. This completes the claim.\\
   We now prove the existence of a $\it{veryweak}$ solution to the problem \eqref{eq5.1p<1} in the sense of Definition \ref{def5.1 VW}. Let us consider $$G:L^1(\Omega) \rightarrow L^1(\Omega)\text{ such that }G(v):=u,$$
   where $v \in L^1(\Omega)$ is assigned to a {\it{veryweak}} solution $u \in L^1(\Omega)$ to the problem \eqref{eq5.6G}. The map $G$ is well-defined. Indeed from the Theorem \ref{thm4.3 DS}, we have the existence of a unique $\it{veryweak}$ solution to the problem \eqref{eq5.6G} (see also Corollary \ref{RM5.1}).\\
\noindent \textbf{Claim 2:} The map $G$ is continuous and relatively compact in $L^1(\Omega)$.\\
 Consider a sequence $(v_n)$  converging to $v$ in $L^1(\Omega)$. Since $t^{-\gamma}$ is decreasing, we get $|\bar g(x,v_n)|\leq f\underline u^{-\gamma}.$ Thus by the Lebesgue dominated convergence theorem, we have
 $$\|(\bar g(x,v_n)+\mu)-(\bar g(x,v)+\mu)\|_{L^1(\Omega)}=\|\bar g(x,v_n)- \bar g(x,v)\|_{L^1(\Omega)} \rightarrow 0 \text{ as }  n \rightarrow \infty.$$ 
 The uniqueness of {\it{veryweak}} solution to the problem \eqref{eq5.6G}, asserts that $G(v_n):=u_n$ converges to $G(v):=u.$ This proves the continuity of $G$ in $L^1(\Omega)$. Now for every $v\in L^1(\Omega)$, we have
 $$\|\bar g(x,v)+\mu\|_{\mathcal{M}(\Omega)}\leq \|\bar g(x,v)\|_{\mathcal{M}(\Omega)}+\|\mu\|_{\mathcal{M}(\Omega)}\leq \left\|\frac{f}{\underline u^\gamma}\right\|_{L^1(\Omega)}+\|\mu\|_{\mathcal{M}(\Omega)}.$$
Therefore, due to Oliva and Petitta \cite[Lemma 2.6]{OP2018} (see also V\'eron \cite[Corollary 2.8]{V2004}), we get $G(v)=u$ is bounded in $ X_0^{s,p}(\Omega)$ for every $q<\frac{N}{N-1}$. On using the compact embedding $X_0^{s,p}(\Omega)\hookrightarrow L^1(\Omega)$, we obtain $G(L^1(\Omega))$ is bounded. Thus, we get $G$ is relatively compact in $L^1(\Omega)$. Therefore, by Schauder's fixed point theorem, we get $G$ has a fixed point $u\in L^1(\Omega)$ of $G$. Finally, using the results of \textbf{\textit{Claim 1}}, we conclude that $u$ is a {\it{veryweak}} solution to the problem \eqref{eq5.1p<1} such that $\underline u \leq u \leq \bar u$. This completes the proof.
\end{proof}
 \begin{theorem}\label{thm5.2}
     The problem \eqref{eq5.1p<1} possesses a $\it{veryweak}$ solution in the sense of the Definition \ref{def5.1 VW}.
 \end{theorem}
 \begin{proof}
    We establish the existence result as a consequence of Theorem \ref{thm5.1} by obtaining a {\it{veryweak subsolution}} and {\it{veryweak supersolution}} to the problem \eqref{eq5.1p<1} in the sense of Definition \ref{def5.2SSS}. We first proceed to obtain a {\it{veryweak subsolution}}. Consider the following problem with singularity,
    \begin{align}\label{eq5.10 WM}
   -\Delta v+(-\Delta)^s & v = \frac{f(x)}{v^\gamma} \,\, \text{in} \,\,\Omega,\\ 
  & v=0\,\, \text{in}\,\,\mathbb  R^N \setminus \Omega. \nonumber
  \end{align}
Proceeding with the steps as in \textbf{\textit{Claim 2}} of Theorem \ref{thm5.1} (see also Theorem \ref{thm3.1<1}), we guarantee the existence of a solution $v\in L^1(\Omega)$ to the problem \eqref{eq5.10 WM}. Let $\phi_1 >0$ be the principle Dirichlet eigenfunction of $(-\Delta +(-\Delta)^s)$ corresponding to the principle eigenvalue $\lambda_1>0$ in $\Omega$. Therefore, we have $\phi_1 |_{\mathbb R^N\setminus\Omega}=0$, $\phi_1>0$ and $\phi_1\in C_{loc}^{2,\alpha}(\Omega)\cap C^{1,\alpha}(\bar{\Omega})$ \cite[Theorem B.1]{SVWZ2023}. Since, $\phi_1>0$ and $v$ is a solution to \eqref{eq5.10 WM}, we choose  $\epsilon >0$, sufficiently small such that
  $$-\Delta (\epsilon \phi)+(- \Delta)^s (\epsilon \phi) -\frac{f}{(\epsilon \phi)^\gamma}<0 =  -\Delta v+(-\Delta)^s v - \frac{f}{v^\gamma}~\text{in the weak sense}.$$ Thus by the weak comparison principle, Lemma \ref{eq6.wcp}, we have $v\geq \epsilon\phi_1>0$ in $\Omega$. We now consider the following approximating problems, 
 \begin{align}\label{lll}
     -\Delta v_n + (-\Delta)^s v_n= \frac{f_n}{(v_n+\frac{1}{n})^\gamma}~ \text{in}~ \Omega, \\ 
      v_n=0 ~ \text{on} ~ \mathbb R^N \setminus \Omega, \nonumber
 \end{align}
 where $f_n=T_n(f)$. By Lemma \ref{lmm2.6}, we get the sequence $(v_n)$ is increasing and converges to a weak solution $v$ of the problem \eqref{eq5.10 WM} with $v_n\in L^\infty(\Omega)$ and $\frac{f_1}{(v_1+1)^\gamma}\in L^\infty(\Omega)$, (see also Theorem \ref{thm3.1<1}). Now on applying the inequality \eqref{Brezis inq}, we get 
 $$\frac{v_1(x)}{\delta(x)}\geq C \int_\Omega \frac{\delta(y)f_1(y)}{(v_1+1)^\gamma} \d y \geq  C \int_\Omega \frac{\delta(y)f_1(y)}{\|v_1\|_{L^\infty(\Omega)}+1} \d y\geq C>0, $$
 which implies
 $$v(x)\geq v_1(x)\geq C\delta(x) ~a.e ~\text{in}~ \Omega.$$
 Therefore, we have $0\leq fv^{-\gamma}\leq Cf\delta(x)^{-\gamma}$. Since, $f\in C^\beta(\bar{\Omega})$, we have $fv^{-\gamma}\in L^1(\Omega)$ for every $0<\gamma <1$. Thus, we obtain 
 $$\int_\Omega v(-\Delta)\phi+\int_\Omega v(-\Delta)^s\phi=\int_\Omega \frac{f\phi}{v^\gamma},\,\,\,\, \forall ~\phi \in C^2_0(\bar\Omega),$$
such that $v>0$ in $\Omega$ and $fv^{-\gamma}\in L^1(\Omega)$. Using the fact that $\mu$ is a nonnegative, bounded Radon measure, we get
  $$\int_\Omega v(-\Delta)\phi+\int_\Omega v(-\Delta)^s\phi\leq \int_\Omega \frac{f\phi}{v^\gamma}+\int_\Omega \phi d\mu,\,\,\,\, \forall~ \phi \in C^2_0(\bar\Omega),\, \phi\geq 0.$$
  Therefore, we conclude that $v$ is a {\it{veryweak subsolution}} to the problem \eqref{eq5.1p<1}.\\
  Next we establish the existence of a {\it{veryweak supersolution}} to the problem \eqref{eq5.1p<1}. Consider the following problem, 
   \begin{equation}\label{eq5.12M}
\begin{aligned}
   -\Delta z+(-\Delta)^s & z =\mu \,\, \text{in} \,\,\Omega,\\
  & z=0\,\, \text{in}\,\,\mathbb  R^N \setminus \Omega,
   \end{aligned}
\end{equation}
where $\mu$ is nonnegative, bounded Radon measure. Thus from Corollary \ref{RM5.1}, we get the existence of a nonnegative {\it{veryweak}} solution $z\in L^1(\Omega)$ to the problem \eqref{eq5.12M}. We set, $w=v+z$, where $v$ and $z$ are {\it{veryweak}} solutions to the problems \eqref{eq5.10 WM} and \eqref{eq5.12M}, respectively. Therefore, we get
\begin{align}\label{eq5.13}
    \int_\Omega w(-\Delta) \phi+\int_\Omega w(-\Delta)^s \phi  & = \int_\Omega v(-\Delta)\phi+ \int_\Omega v(-\Delta)^s\phi+\int_\Omega z(-\Delta)\phi+\int_\Omega z(-\Delta)^s\phi \nonumber \\ 
    & = \int_\Omega \frac{f\phi}{v^\gamma}+\int_\Omega \phi~ d\mu,~\forall\,\phi \in C_0^2(\bar \Omega).
\end{align}
Since, $z\geq0$, $f>0$ and $v>0$, we have $w>0$. Moreover, $fw^{-\gamma}\leq fv^{-\gamma}\in L^1(\Omega)$. Thus from \eqref{eq5.13}, we deduce
\begin{align}
    \int_\Omega\frac{f\phi}{w^\gamma}+\int_\Omega \phi~ d\mu & \leq  \int_\Omega \frac{f\phi}{v^\gamma}+\int_\Omega \phi~ d\mu \nonumber  = \int_\Omega w(-\Delta) \phi+\int_\Omega w(-\Delta)^s \phi,~\forall\, \phi\in C_0^2(\bar \Omega),~\phi \geq 0.
\end{align}
Therefore, we have produced a $w\in L^1(\Omega)$ such that $w>0$, $fw^{-\gamma}\in L^1(\Omega)$ and
$$ \int_\Omega w(-\Delta) \phi+\int_\Omega w(-\Delta)^s \phi \geq  \int_\Omega\frac{f\phi}{w^\gamma}+\int_\Omega \phi~ d\mu,~\forall\, \phi\in C_0^2(\bar \Omega),\,\phi \geq 0.$$
Hence, we conclude that $w$ is a {\it{veryweak supersolution}} to the problem $\eqref{eq5.1p<1}$ such that $w>v$. Therefore, Theorem \ref{thm5.1} asserts that there exists a  $\it{veryweak}$ solution $u\in L^1(\Omega)$ to the problem $\eqref{eq5.1p<1}$ in the sense of Definition \ref{def5.1 VW}. This completes the proof.
 \end{proof}  
In the next theorem, we let out the assumption on $f\in C^\beta(\bar \Omega)$, $0<\beta<1$ with $f\in L^1(\Omega)\cap L^\infty(\Omega_\eta)$ and prove the existence result, where  $\Omega_\eta=\{x\in \Omega:dist(x,\partial \Omega)<\eta\}$ for $\eta>0$ (fixed).
 \begin{theorem}\label{thm5.3}
     Let $f\in L^1(\Omega)\cap L^\infty(\Omega_\eta)$ such that $f>0$ {\it{a.e.}} in $\Omega$ for some fixed $\eta>0$. Then there exists a {\it{veryweak}} solution to the problem $\eqref{eq5.1p<1}$ in the sense of Definition $\ref{def5.1 VW}$.
 \end{theorem}
 \begin{proof}
     Let us consider the following approximating problems:
     \begin{equation}
    \begin{aligned}\label{eq5.15 WM}
    -\Delta v_n +(-\Delta)^s v_n&=\frac{f_n}{(v_n + \frac{1}{n})^{\gamma}}\,\,  \text{in} \,\,\Omega,\\
     v_n &=0 \,\, \,\text{in} \,\, \mathbb R^N \setminus \Omega,
   \end{aligned}
   \end{equation}
    where $f_n=T_n(f)$. Proceeding with similar arguments as in Theorem \ref{thm5.2}, we conclude that the sequence $(v_n)$ is increasing and converges to a weak solution to the problem \eqref{eq5.15 WM} with $v_n\in L^\infty(\Omega)$ and $\frac{f_1}{(v_1+1)^\gamma}\in L^\infty(\Omega)$. Therefore, applying the inequality \eqref{Brezis inq}, we get 
 $$\frac{v_1(x)}{\delta(x)}\geq C \int_\Omega \frac{\delta(y)f_1(y)}{(v_1+1)^\gamma} \d y \geq  C \int_\Omega \frac{\delta(y)f_1(y)}{\|v_1\|_{L^\infty(\Omega)}+1} \d y\geq C>0.$$
 Thus we have $$v(x)\geq v_1(x)\geq C\delta(x), a.e ~\text{on}~ \Omega.$$
 As $f\in L^\infty({\Omega_\eta})$, we conclude that $fv^{-\gamma}\leq Cf\delta(x)^{-\gamma}\in L^1(\Omega)$ for all $\gamma\in(0,1)$. Thus, the {\it{veryweak subsolution}} is bounded from below. This allows us to proceed as in the proof of Theorem \ref{thm5.2} to obtain the existence of a {\it{veryweak}} solution to the problem \eqref{eq5.1p<1}. This completes the proof.
 \end{proof}
 \iffalse
 \begin{remark}
     We mention here that if $f\in C^{\alpha}(\Omega)$ and $0<\gamma<1$, then every nonnegative weak solution $u$ of the problem \eqref{eq5.10 WM} is classical solution. Indeed, from Lemma 3.1,  with $\mu=0$, we have an increasing sequence of solutions $(u_n)$ to \eqref{lll} such that $u_n\in X_0^{s,2}(\Omega)\cap L^{\infty}(\Omega)$. Therefore, by the monotone convergence theorem and proceeding as Theorem , one can prove the existence of a weak solution $u\in X_0^{s,2}(\Omega)\cap L^{\infty}(\Omega)$ to the problem \eqref{eq5.10 WM}. Thus, the regularity $u\in C^{2,\alpha}(\Omega)$ follows from \cite[Theorem 1.4]{SVWZ2023}, which implies $u$ is a classical solution to \eqref{eq5.10 WM}.
 \end{remark}
 \fi
 \noindent \textit{Proof of Theorem \ref{t1.3}}: The proof follows from Theorem \ref{thm5.2} and Theorem \ref{thm5.3}.\\
 \section{Some important results}\label{sec6}
This section is devoted to establishing some important results like Kato-type inequality, weak comparison principle etc. which are essential to guarantee the existence of solutions. We begin with the following Kato-type inequality for the mixed operator $-\Delta +(-\Delta)^s$.
 
 %\subsection{\textbf{Kato-type inequality for mixed local and nonlocal operator}}
 \begin{theorem}[Kato-type inequality]\label{thm6.1}
 Let $f \in L^1(\Omega)$ and $u$ be a {\it{veryweak}} solution to the following problem. Then 
 \begin{equation}\label{eq6.1 K}
\begin{aligned}
   -\Delta u+(-\Delta)^s & u \leq f \,\, \text{in} \,\,\Omega,\\
  & u=0\,\, \text{in}\,\,\mathbb  R^N \setminus \Omega,
   \end{aligned}
\end{equation}
implies $$ (-\Delta +(-\Delta)^s) u^+ \leq (sign_+u)f,$$ in the {\it{veryweak}} sense, where  $sign_+u=\chi_{\{x\in \Omega:u(x) \geq 0\}}$, that is 
\begin{equation}\label{eq6.2}
    \int_\Omega u^+ ( (-\Delta +(-\Delta)^s))\psi \leq \int_\Omega (sign_+u) f \psi, ~ \forall~\psi \in C_0^2(\bar\Omega)~\text{such that}~ \psi \geq 0.
\end{equation}
 \end{theorem}
 \begin{proof}
     Let $\epsilon \in (0,1)$. Consider the function $\phi_\epsilon(s)=(s^2+\epsilon^2)^\frac{1}{2} - \epsilon$. Observe that for each $\epsilon>0$, $\phi_\epsilon(s)$ is convex, Lipschitz function such that $\phi_\epsilon(0)=0$, $\phi'_\epsilon(0)=0$ and $\phi'_\epsilon(s)$ is bounded. Moreover, as $\epsilon \rightarrow 0$, $\phi_\epsilon(s) \rightarrow |s|$ uniformly and $\phi'_\epsilon(s)\rightarrow \operatorname{sign}(s)$, where $\operatorname{sign}(s)=\begin{cases}
          1, ~ &\text{if} ~ s >0,\\
        -1, ~ &\text{if} ~ s <0,\\
          0, ~ &\text{if} ~ s =0.
     \end{cases}$\\
     For every $u\in C_0^2(\bar\Omega)$, we have
    \begin{align}\label{eq6.3}
         (-\Delta)^s \phi_\epsilon(u(x)) & = C \int_\Omega \frac{(\phi_\epsilon(u(x)) -\phi_\epsilon(u(y)))}{|x-y|^{N+2s}} \d y \nonumber \\
         & \leq C \int_\Omega \phi'_\epsilon(u(x)) \frac{(u(x)-u(y))}{|x-y|^{N+2s}} \d y \nonumber \\
         & = \phi'_\epsilon(u(x)) (-\Delta)^su(x).
    \end{align}
  Let $\psi \in C_0^2(\bar\Omega)$ such that $ \psi \geq 0$. Thus using \eqref{eq6.3} and the integration by parts, we get
   \begin{align}\label{eq6.4}
       \int_\Omega &\phi_\epsilon(u(x))((-\Delta)+(-\Delta)^s)\psi(x) =\int_\Omega \phi_\epsilon(u(x)) (-\Delta)\psi(x) +\int_\Omega \phi_\epsilon(u(x)) (-\Delta)^s)\psi(x)  \nonumber \\ 
       %&=\int_\Omega \nabla (\phi_\epsilon(u(x))).\nabla \psi(x) + \int_\Omega(-\Delta)^\frac{s}{2}(\phi_\epsilon(u(x))).(-\Delta)^\frac{s}{2}\psi(x) \nonumber \\ 
       =&\int_\Omega \psi(x) (-\Delta (\phi_\epsilon(u(x))))+\int_\Omega \psi(x) (-\Delta)^s(\phi_\epsilon(u(x))) \nonumber\\ 
       =&-\int_\Omega \psi(x)\phi''_\epsilon(u(x))|\nabla u(x)|^2  + \int_\Omega\psi(x) \phi'_\epsilon(u(x))(-\Delta)u(x) +\int_\Omega\psi(x)(-\Delta)^s(\phi_\epsilon(u(x))) \nonumber \\ 
       &\leq \int_\Omega \psi(x)\phi'_\epsilon(u(x))(-\Delta)u(x)  +\int_\Omega \psi(x)\phi'_\epsilon(u(x)) (-\Delta)^su(x) ~ \nonumber \\
       =&\int_\Omega \psi (x)\phi'_\epsilon(u(x)) ((-\Delta)+(-\Delta)^s)u(x)=\int_\Omega u(x) ((-\Delta)+(-\Delta)^s)(\psi (x)\phi'_\epsilon(u(x))) \nonumber\\
       &\leq \int_\Omega f(x) \psi (x)\phi'_\epsilon(u(x)).
   \end{align}
   Therefore, passing the limit as $\epsilon \rightarrow 0$ in \eqref{eq6.4}, we obtain 
   
   \begin{align}\label{eq6.5k}
       \int_\Omega |u|((-\Delta)+(-\Delta)^s)\psi(x) \d x & \leq \int_\Omega sign(u) f(x) \psi(x)\d x.
   \end{align}
   Hence, using $u^+ \leq |u|$ and $sign(u) \leq sign_+(u)$ in \eqref{eq6.5k}, we get
   \begin{equation}
    \int_\Omega u^+ ( (-\Delta +(-\Delta)^s))\psi \leq \int_\Omega (sign_+u) f \psi, ~ \forall\, \psi\in C_0^2(\bar\Omega)~\text{such that}~ \psi \geq 0.\nonumber
\end{equation}
This completes the proof.
 \end{proof}
 Now using Theorem \ref{thm6.1}, we obtain the following Kato-type inequality for our problem \eqref{eq5.1p<1}.
 \begin{lemma}
 Let $u$ be a $\it{veryweak}$ solution and $\bar u$ be a {\it{veryweak supersolution}} to the problem \eqref{eq5.1p<1}. Then for $f>0$, we have
 $$\int_\Omega(u- \bar u)^+  \leq \int_\Omega \bigg(\frac{1}{u_1^\gamma} -\frac{1}{\bar u^\gamma}\bigg)(sign_+(u-\bar u))f\phi_0,$$
where $\phi_0\in C_0^2(\bar \Omega)$ such that $((-\Delta)+(-\Delta)^s)\phi_0=1$ in $\Omega$ and $\phi=0$ in $\mathbb R^N \setminus \Omega.$
\end{lemma}
\begin{proof}
Recall Definition \ref{def5.1 VW} and Definition \ref{def5.2SSS} for {\it{veryweak}} solution and {\it{veryweak supersolution}} to the problem \eqref{eq5.1p<1}, respectively. Thus on subtracting \eqref{eq5.4sup} from \eqref{eq5.2 Vw}, we get
\begin{equation}\label{eq6.9}
\int_\Omega(u- \bar u) (-\Delta)\phi+\int_\Omega (u-\bar u) (-\Delta)^s\phi \leq \int_\Omega 
\bigg(\frac{1}{u^\gamma} -\frac{1}{\bar u^\gamma}\bigg)f\phi,~\forall~\phi\in C_0^2(\bar \Omega),\,\phi\geq0.
\end{equation}
Note that the equation \eqref{eq6.9} is the $\it{veryweak}$ formulation to the following problem:
\begin{align}\label{eq6.10}
-\Delta (u-\bar u)+(-\Delta)^s & (u-\bar u) \leq \bigg( \frac{1}{u^\gamma}-\frac{1}{\bar u^\gamma}\bigg)f \,\, \text{in} \,\,\Omega,\\ 
  & (u-\bar u) =0\,\, \text{in}\,\,\mathbb  R^N \setminus \Omega.\nonumber
  \end{align}
Therefore, using the Kato-type inequality \eqref{eq6.2} in Theorem \ref{thm6.1} to the problem \eqref{eq6.10}, we deduce that
\begin{equation}\label{eq6.11}
    \int_\Omega(u- \bar u)^+ ((-\Delta)+(-\Delta)^s)\phi \leq \int_\Omega \bigg(\frac{1}{u^\gamma} -\frac{1}{\bar u^\gamma}\bigg)(sign_+(u-\bar u))f\phi,
    \end{equation}
where $sign_+(u-\bar u)=\chi_{\{x\in \Omega:u(x) \geq \bar u(x)\}}$. Choosing $\phi_0\in C_0^2(\bar \Omega)$ such that $((-\Delta)+(-\Delta)^s)\phi_0=1$ in $\Omega$ and $\phi=0$ on $\mathbb R^N \setminus \Omega$ together with \eqref{eq6.11}, we obtain
\begin{align}\label{keto inq 2}
\int_\Omega(u- \bar u)^+  \leq \int_\Omega \bigg(\frac{1}{u^\gamma} -\frac{1}{\bar u^\gamma}\bigg)(sign_+(u-\bar u))f\phi_0,
\end{align}
which is the required Kato-type inequality. This completes the proof.
\end{proof}
\noindent We now prove a maximum principle for the mixed operator $-\Delta + (-\Delta)^s$ followed by the Lemma \ref{lm 6.4} which ensures the integrability of a solution for obtaining the {\it{veryweak}} solution. We proceed with defining the notion of {\it{mixed-superharmonic}} function.
\begin{definition}\label{def6.1}
We say $v$ is a ``{\it{mixed-superharmonic}}" to the following  Dirichlet problem \eqref{f last eq},
     \begin{equation}\label{f last eq}
        \begin{aligned}
            -\Delta v+ (-\Delta)^s v & \geq 0 ~ \,\, \text{in}~\,\, \Omega,\\
            v & = 0 ~ \, \text{in}~\,\, \mathbb R^N \setminus \Omega,
        \end{aligned}
    \end{equation}
      if $v=0$ in $\mathbb R^N\setminus\Omega$ and $v$ satisfies 
    \begin{equation}
        \int_\Omega v (-\Delta) \phi +\int_\Omega v (-\Delta)^s \phi \geq 0, ~ \,\, \forall~\phi\in C_0^2(\bar\Omega)~\text{such that}~\phi\geq 0.
    \end{equation}
    Similarly, one can define {\it{mixed-subharmonic}} functions. A function $v$ is said to be {\it{mixed-harmonic}} if it is {\it{mixed-subharmonic}} and {\it{mixed-superharmonic}}.
\end{definition}

\begin{lemma}[Maximum principle]\label{eq6.15A}
    If $v\in L^1(\Omega)$ be a $\it{veryweak}$ solution to the following
     problem,
     \begin{align}\label{eq6.14A}
          -\Delta v+ (-\Delta)^s v & \geq 0 ~ \,\, \text{in}~\,\, \Omega,\\
            v & = 0 ~ \, \text{in}\,\,~ \mathbb R^N \setminus \Omega, \nonumber
     \end{align}
     then we have $v\geq0$ {\it{a.e.}} in $\Omega$.
     \end{lemma}
     \begin{proof}
         The $\it{veryweak}$ formulation of the problem \eqref{eq6.14A} is given by
         \begin{equation}\label{eq6.18}
        \int_\Omega v (-\Delta) \phi +\int_\Omega v (-\Delta)^s \phi \geq 0, ~ \,\, \forall~\phi\in C_0^2(\bar\Omega)~\text{such that}~\phi\geq 0.
    \end{equation}
    Let $w\in C_0^2(\bar\Omega)$ be a solution to the following problem:
    \begin{equation}\label{eq6.19}
        \begin{aligned}
            -\Delta w+ (-\Delta)^s w & =f ~ \,\, \text{in}~\,\, \Omega,\\
            w & = 0 ~ \, \text{in}~\,\, \mathbb R^N \setminus \Omega,
            \end{aligned}
            \end{equation}
            for every $f\in C^\infty(\bar\Omega)$, (refer to \cite[Theorem 1.4]{SVWZ2023}).
            Now if $f\geq 0$, then whenever $w\geq 0$, we have $w$ is {\it{mixed-subharmonic}} to \eqref{eq6.19}. Again, using \eqref{eq6.19}, we deduce
             \begin{equation}\label{eq6.19A}
        \int_\Omega w (-\Delta) \phi +\int_\Omega w (-\Delta)^s \phi = \int_\Omega f\phi, ~ \,\, \forall~\phi\in C_0^2(\bar\Omega)~\text{with}~\phi\geq 0.
    \end{equation}
   Thus choosing $\phi=v$ as the test function and using the integration by parts, we obtain
    \begin{equation}\label{eq6.21A}
        \int_\Omega v (-\Delta) w +\int_\Omega v (-\Delta)^s w = \int_\Omega f v.
    \end{equation}
    Therefore, using $\eqref{eq6.18}$, we get
    \begin{equation}\label{eq6.21}
        \int_\Omega f v \geq 0, ~ \forall~ f\in C^\infty(\bar\Omega)~ \text{such that}~f\geq 0.
    \end{equation}
    Hence, $v\geq0$ {\it{a.e.}} in $\Omega$. Indeed, consider a bounded, nonnegative sequence $(f_n)\subset C^\infty(\bar\Omega)$ such that $f_n\rightarrow\chi_{\{v<0\}}$ {\it{a.e.}} in $\Omega$. On using the Lebesgue dominated convergence theorem in \eqref{eq6.21}, we get $$0\leq \int_{\{v<0\}} v \leq 0.$$
    Thus, $\int_{\{v<0\}} v=0,$ completing the proof.
     \end{proof}
     
\begin{lemma}\label{lm 6.4}
    Let $0\leq h \in L^\infty(\Omega)$ and let $v$ be a $\it{veryweak}$ solution to the following problem,
    \begin{equation}\label{last eq}
        \begin{aligned}
            -\Delta v+ (-\Delta)^s v & =h ~ \,\, \text{in}~\,\, \Omega,\\
            v & = 0 ~ \, \text{in}~\,\, \mathbb R^N \setminus \Omega.
        \end{aligned}
    \end{equation}
  Then there exists a $C>0$ such that
  \begin{equation}\label{Brezis inq}
      \frac{v(x)}{\delta(x)}\geq C \int_\Omega h(y)\delta(y)\d y~\,\, \forall~ x\in \Omega,\end{equation}
  where $\delta(x)=\operatorname{dist}(x,\partial \Omega).$ 
\end{lemma}
\begin{proof}
We follow Brezis and Cabr\'e \cite[Lemma 3.2]{BC1998}. We divide the proof into two steps. In the first part, we prove the result on every compact subset of $\Omega$. Next using the maximum principle Lemma \ref{eq6.15A}, we prove it for $\Omega$.\\
\textbf{Step I:} For each compact subset $K\subset \Omega$, we have 
    \begin{equation}\label{eq6.20}
    v(x)\geq C \int_\Omega h \delta,~\forall~ x\in K.
    \end{equation}
   Define, $r=\frac{1}{2}\operatorname{dist}(K,\partial \Omega)$. Since $K$ is compact, it is totally bounded, then there exists $m$ balls of radius $r$ such that $x_1,x_2,.....x_m \in K$ and
    $$K \subset B_r(x_1)\cup B_r(x_2) \cup ..... \cup B_r(x_m) \subset \Omega.$$
    Let $u_1,u_2,....u_m$ be the $\it{veryweak}$ solutions to the following problem, 
     \begin{equation}\label{eq6.18L}
        \begin{aligned}
            -\Delta u_i+ (-\Delta)^s u_i & =\chi_{B_r(x_i)} ~ \,\, \text{in}~\,\, \Omega,\\
            u_i & = 0 ~ \, \text{in}~\,\, \mathbb R^N \setminus \Omega,
        \end{aligned}
    \end{equation}
    where $\chi_A$ denotes the characteristic function of a set $A$. By Hopf boundary lemma \cite{AC2023}, there exists a constant $C>0$ such that $$u_i(x)\geq C\delta(x)^s~\forall~ x\in \Omega,~\forall~ 1\leq i\leq m.$$
   Suppose $x\in K$. Then there exists $x_i\in K$ such that $x\in B_r(x_i) \subset B_{2r}(x)\subset \Omega$. Therefore,  using the fact $ -\Delta v+ (-\Delta)^s v \geq 0$ in $\Omega$ and the integration by parts, we get
    \begin{align*}
        v(x)&\geq C\int_{B_{2r}(x)} v \\
        &\geq C\int_{B_{r}(x_i)} v\\
        &= C \int_\Omega u_i(-\Delta+ (-\Delta)^s) v\\ 
        &= C \int_\Omega v(-\Delta+ (-\Delta)^s) u_i\\
        &= C \int_\Omega u_i h \\ 
        & \geq  C \int_\Omega \delta(x)^s h. 
    \end{align*}
  Therefore, for any compact set $K\subset \Omega$, we have $v(x)\geq C \int_\Omega h \delta^s,~ \forall~ x\in K.$\\  
  \textbf{Step II:} Let us fix a compact set $K\subset \Omega$ with smooth boundary $\partial K$. From the equation \eqref{eq6.20}, we get $v(x)\geq C \int_\Omega h \delta^s$ in $K$. Thus, it remains to prove   \eqref{Brezis inq} in $\Omega \setminus K.$ Now consider the problem, 
 \begin{equation}\label{eq6.22}
        \begin{aligned}
            -\Delta w+ (-\Delta)^s w & =0 ~ \,\, \text{in}~\,\, \Omega \setminus K,\\ 
            w & = 0 ~ \, \text{in}~\,\, \mathbb R^N \setminus \Omega, \\ 
            w & =1 ~\, \text{on}~\,\, \partial K.
        \end{aligned}
    \end{equation}
    Let $w$ be a $\it{veryweak}$ solution to \eqref{eq6.22}. Then by the Hopf boundary lemma \cite{AC2023}, we obtain
$$w(x)\geq C \delta(x)^s~\forall\, x\in \Omega \setminus K.$$
From \eqref{last eq} and \eqref{eq6.22}, we deduce
\begin{align}\label{eq6.23}
     -\Delta (v-\lambda w)+ (-\Delta)^s(v-\lambda w) & \geq 0 ~ \,\, \text{in}~\,\, \Omega \setminus K,\\ 
            (v-\lambda w) & = 0 ~ \, \text{in}~\,\, \mathbb R^N \setminus \Omega,\nonumber
\end{align}
where $\lambda=C\bigg(\int_\Omega h \delta^s \bigg)$. Since $v$ is {\it{mixed-superharmonic}} in the sense Definition \ref{def6.1}, we get $v(x)\geq \lambda w(x)$ on $\partial K$. Therefore, applying Lemma \ref{eq6.15A} to the problem \eqref{eq6.23} and using $\delta^s(x)>\delta(x)$, we get
\begin{align}\label{eq6.24}
v(x)-\lambda w(x) &\geq 0~ \forall\, x \,\in \Omega\setminus K. \nonumber
\end{align}
Thus, we have $v(x) \geq C \lambda\delta(x)^s \geq C\lambda \delta(x),~\forall~x \,\in \Omega\setminus K.$ Therefore, we obtain
\begin{align}
    \frac{v(x)}{\delta(x)}\geq C \int_\Omega h \delta^s \geq C \int_\Omega h \delta,~ \forall~x\in \Omega\setminus K.
\end{align}
This completes the proof.
\end{proof}
\noindent We conclude this section with the following weak comparison principle.
\begin{lemma}\label{eq6.wcp}
    Let $u, v\in X_0^{s,2}(\Omega)$ and $f>0$ in $\Omega$. Suppose 
     \begin{align}\label{eq6.wcpineq}
          -\Delta u+ (-\Delta)^s u -\frac{f}{u^{\gamma}} & \geq -\Delta v+ (-\Delta)^s v -\frac{f}{v^{\gamma}} ~ \,\, \text{weakly in}~\,\, \Omega
            \end{align}
            and $u=v=0$ in $\mathbb R^N \setminus \Omega$. Then $u\geq v$ {\it{a.e.}} in $\Omega$.
     \end{lemma}
     \begin{proof}
         Let $u, v\in X_0^{s,2}(\Omega)$ and $f>0$ in $\Omega$. Since,  $-\Delta u+ (-\Delta)^s u -\frac{f}{u^{\gamma}} \geq -\Delta v+ (-\Delta)^s v -\frac{f}{v^{\gamma}}$, weakly and  $u=v=0$ in $\mathbb R^N \setminus \Omega$, then $\forall~\phi\in  X_0^{s,2}(\Omega)$, we get
\begin{align}\label{eq6.26}
         \int_\Omega (\nabla u -\nabla v)\cdot\nabla \phi+\int_\Omega ((-\Delta)^\frac{s}{2} u - (-\Delta)^\frac{s}{2} v)\cdot(-\Delta)^\frac{s}{2} \phi-\int_\Omega\left( \frac{1}{u^\gamma}- \frac{1}{v^\gamma}\right)f\phi\geq 0.
\end{align}
Choose, $\phi=(v-u)^+$ in \eqref{eq6.26}. Therefore, using $f>0$, we get 
\begin{align}\label{eq6.27}
         \int_\Omega (\nabla u -\nabla v)\cdot\nabla (v-u)^+ +\int_\Omega ((-\Delta)^\frac{s}{2} u - (-\Delta)^\frac{s}{2} v)\cdot(-\Delta)^\frac{s}{2} (v-u)^+\geq 0.
\end{align}
On using $u=u^+-u^-$, we get
\begin{align}\label{eq6.28}
         \int_\Omega (\nabla u -\nabla v)\cdot\nabla (v-u)^+ &=-\int_\Omega \nabla (v - u)\cdot\nabla (v-u)^+\nonumber\\
         &=-\int_\Omega |\nabla (v - u)^+|^2\nonumber\\
         &\leq 0.
\end{align}
We have the algebraic inequality,
$$(A-B)(A^+-B^+)\geq (A^+-B^+)^2.$$
Therefore, putting $A=v(x)-u(x)$ and $B=v(y)-u(y)$, we obtain
\begin{align}\label{eq6.29}
         \int_\Omega ((-\Delta)^\frac{s}{2} u - (-\Delta)^\frac{s}{2} v)\cdot(-\Delta)^\frac{s}{2} (v-u)^+&\leq 0.
\end{align}
Thus, using \eqref{eq6.27}, \eqref{eq6.28} and \eqref{eq6.29} together with \eqref{eq6.26}, we get
\begin{align*}
         0\leq&\int_\Omega (\nabla u -\nabla v)\cdot\nabla (v-u)^+ +\int_\Omega ((-\Delta)^\frac{s}{2} u - (-\Delta)^\frac{s}{2} v)\cdot(-\Delta)^\frac{s}{2} (v-u)^+\\
         &-\int_\Omega\left( \frac{1}{u^\gamma}- \frac{1}{v^\gamma}\right)f(v-u)^+\leq 0.
\end{align*}
 Therefore, we conclude that the Lebesgue measure of $\{x\in\Omega: v(x)>u(x)\}$ is zero. Hence, we have $v\leq u$  {\it{a.e.}} in $\Omega$. This completes the proof.   
     \end{proof}

\section*{Conflict of interest statement}
\noindent On behalf of the authors, the corresponding author states that there is no conflict of interest.
 \section*{Data availability statement}
\noindent Data sharing does not apply to this article as no datasets were generated or analysed during the current study.

\section*{Acknowledgement}
\noindent Souvik Bhowmick would like to thank the Council of Scientific and Industrial Research (CSIR), Govt. of India for the financial assistance to carry out this research work [grant no. 09/0874(17164)/2023-EMR-I]. SG acknowledges the research facilities available at the Department of Mathematics, NIT Calicut under the DST-FIST support, Govt. of India [project no. SR/FST/MS-1/2019/40 Dated. 07.01.2020].

%\bibliographystyle{plain}
%\bibliography{svik}

\begin{thebibliography}{10}
	
	\bibitem{AAB2016}
	B.~Abdellaoui, A.~Attar, and R.~Bentifour.
	\newblock On the fractional $p$-laplacian equations with weight and general
	datum.
	\newblock {\em Adv. Nonlinear Anal.}, 8(1):144--174, 2016.
	
	\bibitem{AGS2018}
	Adimurthi, J.~Giacomoni, and S.~Santra.
	\newblock Positive solutions to a fractional equation with singular
	nonlinearity.
	\newblock {\em J. Differential Equations}, 265(4):1191--1226, 2018.
	
	\bibitem{AGM1998}
	C.~O. Alves, J.~V. Goncalves, and L.~A. Maia.
	\newblock Singular nonlinear elliptic equations in $\mathbb{R}^{N}$.
	\newblock {\em Abstr. Appl. Anal.}, 3(3-4):411--423, 1998.
	
	\bibitem{AC2023}
	C.~A. Antonini and M.~Cozzi.
	\newblock Global gradient regularity and a {H}opf lemma for quasilinear
	operators of mixed local-nonlocal type.
	\newblock {\em arXiv preprint arXiv:2308.06075}, 2023.
	
	\bibitem{AM2014}
	D.~Arcoya and L.~Moreno-M{\'e}rida.
	\newblock Multiplicity of solutions for a {D}irichlet problem with a strongly
	singular nonlinearity.
	\newblock {\em Nonlinear Anal.}, 95:281--291, 2014.
	
	\bibitem{AR2023}
	R.~Arora and V.~D. R{\u{a}}dulescu.
	\newblock Combined effects in mixed local-nonlocal stationary problems.
	\newblock {\em Proc. Roy. Soc. Edinburgh Sect. A}, 2023. (To appear).
	
	\bibitem{BBMP2015}
	B.~Barrios, I.~De~Bonis, M.~Medina, and I.~Peral.
	\newblock Semilinear problems for the fractional {L}aplacian with a singular
	nonlinearity.
	\newblock {\em Open Math.}, 13(1):390--407, 2015.
	
	\bibitem{BH2021}
	M.~Bayrami-Aminlouee and M.~Hesaaraki.
	\newblock A fractional {L}aplacian problem with mixed singular nonlinearities
	and nonregular data.
	\newblock {\em J. Elliptic Parabol Equ.}, 7(2):787--814, 2021.
	
	\bibitem{BDVV2021}
	S.~Biagi, S.~Dipierro, E.~Valdinoci, and E.~Vecchi.
	\newblock Semilinear elliptic equations involving mixed local and nonlocal
	operators.
	\newblock {\em Proc. Roy. Soc. Edinburgh Sect. A}, 151(5):1611--1641, 2021.
	
	\bibitem{BDVV2022}
	S.~Biagi, S.~Dipierro, E.~Valdinoci, and E.~Vecchi.
	\newblock Mixed local and nonlocal elliptic operators: regularity and maximum
	principles.
	\newblock {\em Comm. Partial Differential Equations}, 47(3):585--629, 2022.
	
	\bibitem{BV2023}
	S.~Biagi and E.~Vecchi.
	\newblock Multiplicity of positive solutions for mixed local-nonlocal singular
	critical problems.
	\newblock {\em arXiv preprint arXiv:2308.09794}, 2023.
	
	\bibitem{BO2010}
	L.~Boccardo and L.~Orsina.
	\newblock Semilinear elliptic equations with singular nonlinearities.
	\newblock {\em Calc. Var. Partial Differential Equations}, 37(3-4):363--380,
	2010.
	
	\bibitem{BC1998}
	H.~Brezis and X.~Cabr{\'e}.
	\newblock Some simple nonlinear pde's without solutions.
	\newblock {\em Boll. Unione Mat. Ital. Sez. B Artic. Ric. Mat.}, 8(2):223--262,
	1998.
	
	\bibitem{BMP2007}
	H.~Brezis, M.~Marcus, and A.~C. Ponce.
	\newblock Nonlinear elliptic equations with measures revisited.
	\newblock {\em Ann. of Math. Stud.}, 163:55--109, 2007.
	
	\bibitem{BSM2022}
	S.~Buccheri, J.~V. da~Silva, and L.~H. de~Miranda.
	\newblock A system of local$/$nonlocal $p$-{L}aplacians: the eigenvalue problem
	and its asymptotic limit as $p \rightarrow\infty$.
	\newblock {\em Asymptot. Anal.}, 128(2):149--181, 2022.
	
	\bibitem{BS2023}
	S.~Byun and K.~Song.
	\newblock Mixed local and nonlocal equations with measure data.
	\newblock {\em Calc. Var. Partial Differential Equations}, 62(1):1--35, Article
	No.14, 35 pp, 2023.
	
	\bibitem{CD2004}
	A.~Canino and M.~Degiovanni.
	\newblock A variational approach to a class of singular semilinear elliptic
	equations.
	\newblock {\em J. Convex Anal.}, 11(1):147--162, 2004.
	
	\bibitem{CMSS2017}
	A.~Canino, L.~Montoro, B.~Sciunzi, and M.~Squassina.
	\newblock Nonlocal problems with singular nonlinearity.
	\newblock {\em Bull. Sci. Math.}, 141(3):223--250, 2017.
	
	\bibitem{CS2016}
	A.~Canino and B.~Sciunzi.
	\newblock A uniqueness result for some singular semilinear elliptic equations.
	\newblock {\em Commun. Contemp. Math.}, 18(06):1--8, Article No. 1550084, 9 pp,
	2016.
	
	\bibitem{CST2016}
	A.~Canino, B.~Sciunzi, and A.~Trombetta.
	\newblock Existence and uniqueness for $p$-{L}aplace equations involving
	singular nonlinearities.
	\newblock {\em NoDEA Nonlinear Differential Equations Appl.}, 23(2):8--18,
	2016.
	
	\bibitem{CSYZ2024}
	I.~Chlebicka, K.~Song, Y.~Youn, and A.~Zatorska-Goldstein.
	\newblock Riesz potential estimates for mixed local-nonlocal problems with
	measure data.
	\newblock {\em arXiv preprint arXiv:2401.04549}, 2024.
	
	\bibitem{CRT1977}
	M.~G. Crandall, P.~H. Rabinowitz, and L.~Tartar.
	\newblock On a {D}irichlet problem with a singular nonlinearity.
	\newblock {\em Comm. Partial Differential Equations}, 2(2):193--222, 1977.
	
	\bibitem{DFV2024}
	J.~V. Da~Silva, A.~Fiscella, and V.~A.~B. Viloria.
	\newblock Mixed local-nonlocal quasilinear problems with critical
	nonlinearities.
	\newblock {\em J. Differential Equations}, pages 494--536, 2024.
	
	\bibitem{DMOP1999}
	G.~Dal~Maso, F.~Murat, L.~Orsina, and A.~Prignet.
	\newblock Renormalized solutions of elliptic equations with general measure
	data.
	\newblock {\em Ann. Scuola Norm. Sup. Pisa Cl. Sci. (4)}, 28(4):741--808, 1999.
	
	\bibitem{C2013}
	L.~M. De~Cave.
	\newblock Nonlinear elliptic equations with singular nonlinearities.
	\newblock {\em Asymptot. Anal.}, 84(3-4):181--195, 2013.
	
	\bibitem{DM2022}
	C.~De~Filippis and G.~Mingione.
	\newblock Gradient regularity in mixed local and nonlocal problems.
	\newblock {\em Math. Ann.}, 388(1):261--328, 2022.
	
	\bibitem{PFR2019}
	L.~M. Del~Pezzo, R.~Ferreira, and J.~D. Rossi.
	\newblock Eigenvalues for a combination between local and nonlocal
	$p$-{L}aplacians.
	\newblock {\em Fract. Calc. Appl. Anal.}, 22(5):1414--1436, 2019.
	
	\bibitem{NPV2012}
	E.~Di~Nezza, G.~Palatucci, and E.~Valdinoci.
	\newblock Hitchhiker's guide to the fractional {S}obolev spaces.
	\newblock {\em Bull. Sci. Math.}, 136(5):521--573, 2012.
	
	\bibitem{DPV2023}
	S.~Dipierro, E.~Proietti~Lippi, and E.~Valdinoci.
	\newblock (non)local logistic equations with {N}eumann conditions.
	\newblock {\em Ann. Inst. H. Poincar\'e C Anal. Non Lin\'eaire},
	40(5):1093--1166, 2023.
	
	\bibitem{DV2021}
	S.~Dipierro and E.~Valdinoci.
	\newblock Description of an ecological niche for a mixed local/nonlocal
	dispersal: an evolution equation and a new {N}eumann condition arising from
	the superposition of {B}rownian and {L}\'evy processes.
	\newblock {\em Phys. A}, 575:1093--1166, Paper No. 126052, 20 pp., 2021.
	
	\bibitem{EVANS2022}
	L.~C. Evans.
	\newblock {\em Partial {D}ifferential {E}quations: {S}econd {E}dition}, volume
	19 of Graduate Studies in Mathematics.
	\newblock American Mathematical Society, Providence, RI, 749 pp., 2010.
	
	\bibitem{F1999}
	G.~B. Folland.
	\newblock {\em Real {A}nalysis: {M}odern {T}echniques and {T}heir
		{A}pplications, {S}econdnd {E}dition}, volume~40.
	\newblock John Wiley \& Sons, Inc., New York, 416 pp., 1999.
	
	\bibitem{G2018}
	P.~Garain.
	\newblock On a degenerate singular elliptic problem.
	\newblock {\em Math. Nachr.}, 295(7):1354--1377, 2022.
	
	\bibitem{G2023}
	P.~Garain.
	\newblock On a class of mixed local and nonlocal semilinear elliptic equation
	with singular nonlinearity.
	\newblock {\em J. Geom. Anal.}, 33(7):1--20, Article No. 212, 20 pp., 2023.
	
	\bibitem{GK2022}
	P.~Garain and J.~Kinnunen.
	\newblock On the regularity theory for mixed local and nonlocal quasilinear
	elliptic equations.
	\newblock {\em Trans. Amer. Math. Soc.}, 375(8):5393–--5423, 2022.
	
	\bibitem{GU2022}
	P.~Garain and A.~Ukhlov.
	\newblock Mixed local and nonlocal {S}obolev inequalities with extremal and
	associated quasilinear singular elliptic problems.
	\newblock {\em Nonlinear Anal.}, 223:1--35, Article No. 113022, 2022.
	
	\bibitem{GCG2019}
	S.~Ghosh, D.~Choudhuri, and R.~K. Giri.
	\newblock Singular nonlocal problem involving measure data.
	\newblock {\em Bull. Braz. Math. Soc. (N.S.)}, 50(1):187--209, 2019.
	
	\bibitem{GMM2017}
	D.~Giachetti, P.~J. Mart{\'\i}nez-Aparicio, and F.~Murat.
	\newblock A semilinear elliptic equation with a mild singularity at $u= 0$:
	existence and homogenization.
	\newblock {\em J. Math. Pures Appl. (9)}, 107(1):41--77, 2017.
	
	\bibitem{H2003}
	Y.~Haitao.
	\newblock Multiplicity and asymptotic behavior of positive solutions for a
	singular semilinear elliptic problem.
	\newblock {\em J. Differential Equations}, 189(2):487--512, 2003.
	
	\bibitem{KPU2011}
	K.~H. Karlsen, F.~Petitta, and S.~Ulusoy.
	\newblock A {D}uality approach to the fractional {L}aplacian with measure data.
	\newblock {\em Publ. Mat.}, 55(1):151--161, 2011.
	
	\bibitem{K2002}
	T.~Kilpel{\"a}inen.
	\newblock Problems involving $ p $-{L}aplacian type equations and measures.
	\newblock {\em Math. Bohem.}, 127(2):243--250, 2002.
	
	\bibitem{KMS2015}
	T.~Kuusi, G.~Mingione, and Y.~Sire.
	\newblock Nonlocal equations with measure data.
	\newblock {\em Commun. Math. Phys.}, 337(3):1317--1368, 2015.
	
	\bibitem{K2017}
	M.~Kwa{\'s}nicki.
	\newblock Ten equivalent definitions of the fractional {L}aplace operator.
	\newblock {\em Fract. Calc. Appl. Anal.}, 20(1):7--51, 2017.
	
	\bibitem{LGG2024}
	R.~Lakshmi, R.~Kr. Giri, and S.~Ghosh.
	\newblock A weighted eigenvalue problem for mixed local and nonlocal operators
	with potential.
	\newblock {\em arXiv preprint arXiv:2409.01349}, 2024.
	
	\bibitem{LM1991}
	A.~C. Lazer and P.~J. McKenna.
	\newblock On a singular nonlinear elliptic boundary-value problem.
	\newblock {\em Proc. Amer. Math. Soc.}, 111(3):721--730, 1991.
	
	\bibitem{LPPS2015}
	T.~Leonori, I.~Peral, A.~Primo, and F.~Soria.
	\newblock Basic estimates for solutions of a class of nonlocal elliptic and
	parabolic equations.
	\newblock {\em Discrete Contin. Dyn. Syst.}, 35(12):6031--–6068, 2015.
	
	\bibitem{MV2014}
	M.~Marcus and L.~V{\'e}ron.
	\newblock {\em Nonlinear second order elliptic equations involving measures},
	volume~21.
	\newblock De Gruyter, Berlin, 2014.
	
	\bibitem{OP2016}
	F.~Oliva and F.~Petitta.
	\newblock On singular elliptic equations with measure sources.
	\newblock {\em ESAIM Control Optim. Calc. Var.}, 22(1):289--308, 2016.
	
	\bibitem{OP2018}
	F.~Oliva and F.~Petitta.
	\newblock Finite and infinite energy solutions of singular elliptic problems:
	existence and uniqueness.
	\newblock {\em J. Differential Equations}, 264(1):311--340, 2018.
	
	\bibitem{OP2024}
	F.~Oliva and F.~Petitta.
	\newblock Singular {E}lliptic {PDE}s: an extensive overview.
	\newblock {\em arXiv preprint arXiv:2409.00482}, 2024.
	
	\bibitem{PGC2017}
	A.~Panda, S.~Ghosh, and D.~Choudhuri.
	\newblock Elliptic partial differential equation involving singularity.
	\newblock {\em J. Indian Math. Soc. (N.S.)}, 86(1-2):95--117, 2019.
	
	\bibitem{P2016}
	F.~Petitta.
	\newblock Some remarks on the duality method for integro-differential equations
	with measure data.
	\newblock {\em Adv. Nonlinear Stud.}, 16(1):115--124, 2016.
	
	\bibitem{A2016}
	A.~C. Ponce.
	\newblock {\em Elliptic {PDE}s, {M}easures and {C}apacities: {F}rom the
		{P}oisson {E}quation to {N}onlinear {T}homas-{F}ermi {P}roblems}.
	\newblock European Mathematical Society (EMS), Switzerland, 453 pp., 2016.
	
	\bibitem{S2017}
	K.~Saoudi.
	\newblock A critical fractional elliptic equation with singular nonlinearities.
	\newblock {\em Fract. Calc. Appl. Anal.}, 20(6):1507--1530, 2017.
	
	\bibitem{SGC2019}
	K.~Saoudi, S.~Ghosh, and D.~Choudhuri.
	\newblock Multiplicity and {H}ölder regularity of solutions for a nonlocal
	elliptic pde involving singularity.
	\newblock {\em J. Math. Phys.}, 60(10):1--42, Article No.101509, 28 pp, 2019.
	
	\bibitem{S1995}
	G.~Stampacchia.
	\newblock Le probl{\`e}me de {D}irichlet pour les {\'e}quations elliptiques du
	second ordre {\`a} coefficients discontinus.
	\newblock {\em Ann. Inst. Fourier (Grenoble)}, 15(1):189--257, 1965.
	
	\bibitem{SVWZ2023}
	X.~Su, E.~Valdinoci, Y.~Wei, and J.~Zhang.
	\newblock On some regularity properties of mixed local and nonlocal elliptic
	equations.
	\newblock {\em Preprint available at SSRN: 4617397}, 2023.
	
	\bibitem{F2014}
	D.~Tai and Y.~Fang.
	\newblock Existence and uniqueness of positive solution to a fractional
	{L}aplacians with singular nonlinearity.
	\newblock {\em Appl. Math. Lett.}, 119:1--7, Article No. 107227, 2021.
	
	\bibitem{T1979}
	S.~D. Taliaferro.
	\newblock A nonlinear singular boundary value problem.
	\newblock {\em Nonlinear Anal.}, 3(6):897--904, 1979.
	
	\bibitem{V2004}
	L.~V{\'e}ron.
	\newblock Elliptic {E}quations {I}nvolving {M}easures.
	\newblock In {\em Handbook of Differential Equations: Stationary Partial
		Differential Equations, M. Chipot, P. Quittner. (editors)}, pages 595--712.
	Elsevier, 2004.
	
	\bibitem{YD2014}
	S.~Yijing and Z.~Duanzhi.
	\newblock The role of the power $3$ for elliptic equations with negative
	exponents.
	\newblock {\em Calc. Var. Partial Differential Equations}, 49(3-4):909--922,
	2014.
	
	\bibitem{YM2021}
	A.~Youssfi and G.~O.~M. Mahmoud.
	\newblock Nonlocal semilinear elliptic problems with singular nonlinearity.
	\newblock {\em Calc. Var. Partial Differential Equations}, 60(4):153, 2021.
	
\end{thebibliography}

\end{document}